\documentclass[psamsfonts, 12pt,reqno]{amsart}

\usepackage[backref]{hyperref}
\usepackage{amssymb}
 \numberwithin{equation}{section}
\usepackage{graphicx}
\usepackage{amsmath}
\usepackage{mathrsfs}
\usepackage{latexsym}
\usepackage{mathrsfs}
\usepackage{graphicx}
\usepackage[utf8]{inputenc}
\usepackage[T1]{fontenc}

\usepackage{color}



\theoremstyle{plain}
\newtheorem{theorem}{Theorem}[section]
\newtheorem{proposition}{Proposition}[section]
\newtheorem{lemma}{Lemma}[section]
\newtheorem{corollary}{Corollary}[section]

\newtheorem{definition}{Definition}[section]
\begin{document}

\title[]{Normalized states for a critical Schr\"{o}dinger--Poisson system with Hardy singularity: multiplicity and semiclassical concentration}

\author{Khaled Khachnaoui}

\begin{abstract}
In this paper, we investigate the existence, multiplicity, and
semiclassical concentration of normalized solutions to a critical
Schr\"{o}dinger--Poisson system with a singular Hardy potential in
\(\mathbb{R}^3\). More precisely, we consider
\[
\begin{cases}
-\varepsilon^2\Delta u+
\left(V(x)-\dfrac{\kappa\varepsilon^2}{|x|^2}\right)u -\phi |u|^3u
=\lambda u+\mu |u|^{q-2}u+|u|^4u,
& \text{in } \mathbb{R}^3, \\[1mm]
-\varepsilon^2\Delta\phi=|u|^5, & \text{in } \mathbb{R}^3,
\end{cases}
\]
under the prescribed mass constraint
\[
\int_{\mathbb{R}^3}|u|^2\,dx=a^2\varepsilon^3,
\]
where \(a,\mu>0\), \(q\in(2,10/3)\), \(\varepsilon>0\) is a small
semiclassical parameter, and \(0<\kappa<1/4\). The parameter
\(\lambda\in\mathbb{R}\) appears as a Lagrange multiplier associated
with the mass constraint, while \(V:\mathbb{R}^3\to(0,+\infty)\) is
a continuous electric potential whose minimum set is assumed to be
nonempty and compact. The main difficulty stems from the
simultaneous presence of the inverse-square Hardy singularity, the
mass constraint, the critical local nonlinearity, and the nonlocal
Poisson interaction. By combining the Hardy inequality, constrained
variational methods, suitable truncation arguments, and
concentration-compactness techniques, we first establish the
existence of a normalized ground state for sufficiently small mass
and sufficiently small \(\varepsilon\). We then employ
Ljusternik--Schnirelmann category theory to obtain multiple
normalized solutions whose number is related to the topology of the
minimum set of \(V\). Finally, we show that the corresponding
semiclassical states concentrate near the global minimum set of the
electric potential as \(\varepsilon\to0\).

\end{abstract}


\maketitle

\tableofcontents
\section{Introduction}

Nonlinear Schr\"odinger--Poisson systems arise in several areas of
mathematical physics, including quantum mechanics, semiconductor
theory, plasma physics, and the description of charged particles
interacting with the electrostatic field generated by their own
density \cite{LEH,LEL,LPL}. A general form of these systems is
\[
\begin{cases}
-\varepsilon^2\Delta u+V(x)u-\phi\,g(u)=f(u),
& x\in\mathbb R^3,\\
-\varepsilon^2\Delta\phi=G(u), & x\in\mathbb R^3,
\end{cases}
\]
where $u$ denotes the wave function, $\phi$ is the associated
electrostatic potential, $V$ is an external potential, and
$\varepsilon>0$ is the semiclassical parameter. The coupling between
the Schr\"odinger and Poisson equations gives rise to a nonlocal
interaction, making the associated variational analysis considerably
more involved than that of scalar nonlinear Schr\"odinger equations
\cite{Cazenave}.

The mathematical study of Schr\"odinger--Poisson systems has been
extensively developed since the pioneering work of Benci and
Fortunato \cite{BenciFortunato,BenciFortunato2002}. Subsequently,
numerous authors investigated the existence, multiplicity,
qualitative properties, and semiclassical concentration of
solutions; see, for example,
\cite{AmbrosettiRuiz,Azzollini,AzzolliniDAprize,CeramiVaira,DAprileMugnai,Ruiz,ZhaoZhao}
and the references therein.

In the semiclassical regime, one of the fundamental questions is the
concentration behavior of bound states as $\varepsilon\rightarrow0$.
Variational methods \cite{WMM}, concentration--compactness
principles \cite{Lions1,Lions2,Lions3}, penalization techniques
\cite{delPinoFelmer}, and Lusternik--Schnirelmann theory
\cite{Szulkin} have proved particularly effective for describing
this phenomenon. The presence of a magnetic field adds further
structure; we refer to
\cite{CingolaniSecchi,Diamagnetic,EstebanLions} for semiclassical
states in magnetic Schr\"odinger equations.

More recently, increasing attention has been devoted to normalized
solutions. In this setting, instead of prescribing the frequency,
one fixes the mass of the state by imposing $\int_{\mathbb
R^3}|u|^2\,dx=a^2,$ while the frequency appears as an unknown
Lagrange multiplier. Normalized states are physically relevant since
the $L^2$-norm represents the total mass (or number of particles) of
the system. From the analytical viewpoint, however, the prescribed
mass constraint forces the energy functional to be studied on an
$L^2$-sphere, creating additional compactness difficulties,
especially in the presence of critical nonlinearities. Foundational
results in this direction were obtained by Bartsch and de Valeriola
\cite{BartschValeriola} and by Bartsch and Jeanjean
\cite{BartschJeanjean}.

In the semiclassical regime, the prescribed mass naturally scales
with the concentration parameter. Indeed, if a family of solutions
concentrates around a point $y\in\mathbb R^3$ at the scale
$\varepsilon$, namely $u_\varepsilon(x) =
w\!\left(\frac{x-y}{\varepsilon}\right),$ then $\int_{\mathbb
R^3}|u_\varepsilon|^2\,dx = \varepsilon^3 \int_{\mathbb
R^3}|w|^2\,dx.$ Consequently, the normalization
$\|u_\varepsilon\|_2^2 = a^2\varepsilon^3$ is the natural
semiclassical scaling, ensuring that the limiting profile has
prescribed mass $a$.

Considerable progress has been achieved for critical
Schr\"odinger--Poisson equations. Feng \cite{FX} investigated the
semiclassical problem
\begin{equation}\label{fcv}
\begin{cases}
-\varepsilon^2\Delta u+V(x)u-\phi|u|^3u = f(u)+|u|^4u,
& x\in\mathbb R^3,\\[2mm]
-\varepsilon^2\Delta\phi=|u|^5, & x\in\mathbb R^3,
\end{cases}
\end{equation}
where the perturbation possesses Sobolev-subcritical growth. By
combining a modified concentration--compactness principle with the
mountain-pass theorem, positive ground state solutions were obtained
together with their concentration behavior in the semiclassical
limit.

The normalized counterpart of critical Schr\"odinger--Poisson
systems has also attracted considerable interest. In \cite{VR}, the
authors established existence and multiplicity results for
normalized semiclassical solutions of a Sobolev-critical
Schr\"odinger--Poisson equation involving a critical nonlocal
interaction. By means of Lusternik--Schnirelmann theory, the number
of concentrating solutions was related to the topology of the set
where the electric potential attains its minimum. Related
multiplicity results for normalized solutions via
Lusternik--Schnirelmann category have been obtained in
\cite{ACO,AJC,XCJ}.

Several contributions have also been devoted to normalized
Schr\"odinger--Poisson equations under prescribed mass constraints.
Bartsch and Jeanjean \cite{BartschJeanjean} studied normalized
solutions involving critical nonlinearities, whereas multiplicity
results for related non-magnetic models were obtained by Jeanjean
and Lu \cite{JeanjeanLu} and by Bellazzini, Jeanjean, and Luo
\cite{BellazziniJeanjeanLuo}. The planar case was investigated by
Shu, Wen, and Yang \cite{SMW}. For equations with potentials in
unbounded domains, we refer to Lancelotti and Molle \cite{LSM}, and
for problems involving $(2,q)$-Laplacian operators, to Wei and Song
\cite{WS}. A comprehensive review of normalized solutions for four
classes of nonlinear elliptic equations can be found in \cite{CST}.

A particularly important contribution was obtained by Meng and He
\cite{MY}, who considered the critical normalized system
\begin{equation}\label{10J}
\begin{cases}
-\Delta u-\phi|u|^3u = \lambda u+\mu|u|^{q-2}u+|u|^4u,
& x\in\mathbb R^3,\\[2mm]
-\Delta\phi=|u|^5,
& x\in\mathbb R^3,\\[2mm]
\displaystyle \int_{\mathbb R^3}|u|^2\,dx=a^2.
\end{cases}
\end{equation}
Using genus theory, they established multiplicity results in the
$L^2$-subcritical regime and further investigated the
$L^2$-supercritical case for sufficiently large values of $\mu$.
Further developments on normalized solutions with doubly critical
growth were given in \cite{MengHe2023,MengHe2024}.

Subsequently, He, Meng, and R\u{a}dulescu \cite{HXX} revisited
problem \eqref{10J} with the perturbation $\mu|u|^{q-2}u$, covering
the $L^2$-subcritical, $L^2$-critical, and $L^2$-supercritical
regimes. They obtained existence and nonexistence results and
analyzed the asymptotic behavior of ground states as $\mu\to0^+$.
The fractional analogue of these systems was studied by He, Meng,
and Squassina \cite{HeMengSquassina2024}, while Khachnaoui
\cite{KK,KKa} investigated related nonlocal problems.

More recently, He, Liu, and Meng \cite{HXL} investigated the
generalized normalized problem
\begin{equation}\label{potential_problem}
\begin{cases}
-\Delta u+V(x)u-\phi|u|^3u = \lambda u+\mu|u|^{q-2}u+|u|^4u,
& x\in\mathbb R^3,\\[2mm]
-\Delta\phi=|u|^5,
& x\in\mathbb R^3,\\[2mm]
\displaystyle \int_{\mathbb R^3}|u|^2\,dx=a^2,
\end{cases}
\end{equation}
where the electric potential vanishes at infinity. Their approach,
based on the Pohozaev manifold and constrained variational methods,
covers the $L^2$-subcritical, critical, and supercritical regimes,
and also provides decay estimates for the corresponding ground
states. Related developments may be found in Gao and He \cite{GQH},
Guo and Zhang \cite{GZZ}, Long and Feng \cite{LLF}, Chen and Tang
\cite{CST}, Wei and Song \cite{WS}, Jin, Yang, and Zhou \cite{JPY},
Jin and Zhang \cite{JZZ}, Kang and Tang \cite{KangTang}, and the
references therein.

The inverse-square Hardy potential, which is the subject of the
present work, has been extensively studied in the context of
critical Schr\"odinger equations. Fan, Li, and Tang \cite{FanLiTang}
investigated normalized ground state solutions for critical growth
Schr\"odinger equations with Hardy potential, while Li, Li, and Tang
\cite{LiLiTang} and Li and Zou \cite{LiZou} established ground state
and normalized solutions for related problems. Terracini
\cite{Terracini} studied positive entire solutions to equations with
singular coefficients and critical exponents.

Despite these significant advances, the available literature is
mainly devoted to non-magnetic Schr\"odinger--Poisson equations.
Introducing simultaneously a magnetic field, a prescribed mass
constraint, Sobolev-critical local nonlinearities, critical nonlocal
Poisson interactions, and a semiclassical parameter leads to a
substantially different variational framework. Consequently, the
existing methods cannot be directly applied to the present setting.

Motivated by the above developments, in this paper we investigate
the normalized semiclassical Schr\"odinger--Poisson problem
\begin{equation}\label{P}
\begin{cases}
-\varepsilon^2\Delta u+ \left(
V(x)-\dfrac{\kappa\varepsilon^2}{|x|^2} \right)u -\phi|u|^3u =
\lambda u+\mu|u|^{q-2}u+|u|^4u,
& x\in\mathbb R^3,\\[2mm]
-\varepsilon^2\Delta\phi = |u|^5,
& x\in\mathbb R^3,\\[2mm]
\displaystyle \int_{\mathbb R^3}|u|^2\,dx = a^2\varepsilon^3,
\end{cases}
\tag{$P_\varepsilon$}
\end{equation}
where $a>0$, $\mu>0$, $2<q<\frac{10}{3}$, $0<\kappa<\frac14$, and
$\lambda\in\mathbb R$ is the Lagrange multiplier associated with the
prescribed mass.

The inverse-square Hardy potential plays a distinguished role both
analytically and physically. Indeed, it has the same homogeneity as
the Laplace operator and is governed by the Hardy inequality
\[
\frac14 \int_{\mathbb R^3} \frac{|u|^2}{|x|^2}\,dx \le \int_{\mathbb
R^3} |\nabla u|^2\,dx.
\]
Hence, the restriction $0<\kappa<1/4$ guarantees the coercivity of
the singular quadratic form
\[
\int_{\mathbb R^3} \left( |\nabla u|^2 - \kappa \frac{|u|^2}{|x|^2}
\right)\,dx.
\]
Nevertheless, the Hardy singularity interacts nontrivially with the
mass-preserving scaling and with the critical nonlinearities,
considerably increasing the analytical difficulties.

The analysis of problem \eqref{P} is particularly delicate since it
combines four critical ingredients: a prescribed mass constraint, an
attractive Hardy potential, a Sobolev-critical local nonlinearity,
and a critical nonlocal Poisson interaction. The interplay among
these effects significantly influences the geometry of the
constrained energy functional, the compactness properties, the
Pohozaev identity, and the concentration mechanism in the
semiclassical limit.

The principal novelty of the present work lies in introducing an
attractive Hardy singularity into a normalized semiclassical
Schr\"odinger--Poisson equation involving simultaneous local and
nonlocal critical nonlinearities. To the best of our knowledge, this
combination has not been investigated previously. Our first main
result establishes the existence of a positive normalized ground
state for sufficiently small values of the prescribed mass and the
semiclassical parameter. Our second result proves that the number of
positive normalized solutions is bounded from below by the
Lusternik--Schnirelmann category of the minimum set of the electric
potential and describes their concentration behavior as
$\varepsilon\to0$.

Our approach combines several variational techniques. Hardy's
inequality yields the coercivity of the singular quadratic form,
while suitable Gagliardo--Nirenberg inequalities and refined
estimates for the Poisson term provide the mountain-pass geometry on
the prescribed mass sphere. A suitable compactness threshold is
introduced to exclude bubbling phenomena generated by the
coexistence of the local and nonlocal critical nonlinearities.
Finally, a barycenter map combined with Lusternik--Schnirelmann
category theory transfers the topology of the minimum set of the
potential to the low-energy sublevels of the constrained functional,
leading to the multiplicity and concentration results established in
this paper. Throughout this paper, we assume that $a>0,\qquad
\mu>0,$ $2<q<\frac{10}{3},$  $0<\kappa<\frac14.$ We impose the
following assumptions on the electric potential $V:\mathbb
R^3\to\mathbb R$:

\begin{itemize}
\item[\rm (V1)]
$V\in C^1(\mathbb R^3,\mathbb R)$ and
\[
V_0:=\inf_{x\in\mathbb R^3}V(x)>0.
\]
Moreover, there exists $C_V>0$ such that
\[
|\nabla V(x)\cdot x|\leq C_V \qquad \text{for all }x\in\mathbb R^3.
\]

\item[\rm (V2)]
There exists a bounded open set $\Lambda\subset\mathbb R^3$ such
that
\[
0\notin\overline{\Lambda} \qquad\text{and}\qquad
V_0<\min_{x\in\partial\Lambda}V(x).
\]

\item[\rm (V3)]
The minimum set
\[
\mathcal M := \{x\in\Lambda:V(x)=V_0\}
\]
is nonempty and compact. Moreover,
\[
\operatorname{dist}(\mathcal M,\partial\Lambda)>0.
\]

\end{itemize}

The condition $0\notin\overline{\Lambda}$ ensures that the expected
concentration region remains uniformly separated from the singular
point of the Hardy potential. This geometric assumption allows the
inverse-square singularity to be controlled globally through the
Hardy inequality, while the semiclassical localization takes place
in a regular region of the electric potential.

For every $\delta>0$, we set
\[
\mathcal M_\delta := \left\{ x\in\mathbb R^3:
\operatorname{dist}(x,\mathcal M)\leq\delta \right\}.
\]
Throughout the paper, $\delta>0$ is chosen sufficiently small so
that
\[
\mathcal M_\delta\subset\Lambda \qquad\text{and}\qquad
\operatorname{dist}(\mathcal M_\delta,\{0\})>0.
\]

Our first main result establishes the existence of a positive
normalized ground state.

\begin{theorem}\label{T1}
Assume that {\rm (V1)--(V3)} hold. Then there exist $a_*>0$ and
$\varepsilon_*>0$ such that, for every $a\in(0,a_*),\;
\varepsilon\in(0,\varepsilon_*),$ problem \eqref{P} admits a
normalized solution $u_\varepsilon\in H^1(\mathbb R^3)$ satisfying
\[
\int_{\mathbb R^3}|u_\varepsilon|^2\,dx = a^2\varepsilon^3.
\]
Moreover, $u_\varepsilon$ can be chosen positive in $\mathbb
R^3\setminus\{0\}$ and is a ground state among all normalized
solutions of problem \eqref{P}.
\end{theorem}

Our second main result concerns the multiplicity and concentration
of normalized semiclassical states. The number of solutions is
governed by the topology of the minimum set $\mathcal M$.

\begin{theorem}\label{T2}
Assume that {\rm (V1)--(V3)} hold. For every sufficiently small
$\delta>0$, there exist $a_\delta>0$ and $\varepsilon_\delta>0$ such
that, for every $a\in(0,a_\delta),\;
\varepsilon\in(0,\varepsilon_\delta),$ problem \eqref{P} possesses
at least $\operatorname{cat}_{\mathcal M_\delta}(\mathcal M)$
distinct positive normalized solutions. Furthermore, let
$\varepsilon_n\to0$ and let $u_n$ be any sequence of the solutions
obtained above. If $x_n$ is a global maximum point of $u_n$, then
$\operatorname{dist}(x_n,\mathcal M)\to0.$ Consequently, $V(x_n)\to
V_0.$ Moreover, up to a subsequence, there exists a positive
normalized ground state $w_a$ of the autonomous limit problem such
that
\[
u_n(x_n+\varepsilon_n x) \to w_a(x) \qquad \text{strongly in
}H^1(\mathbb R^3).
\]
\end{theorem}

\section{Analytical Tools and Functional Setting}
\label{sec:functional}

In this section, we introduce the variational framework associated
with problem \eqref{P}. We also record the scaling properties and
the Pohozaev functional which will be used later in the compactness
analysis. Particular attention is paid to the Hardy singularity and
to the critical Poisson interaction.

\subsection{The Hardy inequality and the energy space}

We first recall the Hardy inequality in dimension three:
\begin{equation}\label{Hardy}
\frac14\int_{\mathbb R^3}\frac{|u|^2}{|x|^2}\,dx \leq \int_{\mathbb
R^3}|\nabla u|^2\,dx, \qquad u\in H^1(\mathbb R^3).
\end{equation}
Since $0<\kappa<\frac14,$ we deduce
\begin{equation}\label{Hardy-coercivity}
\int_{\mathbb R^3} \left( |\nabla u|^2-\kappa\frac{|u|^2}{|x|^2}
\right)\,dx \geq (1-4\kappa)\int_{\mathbb R^3}|\nabla u|^2\,dx .
\end{equation}
Thus the singular quadratic form associated with
$-\Delta-\frac{\kappa}{|x|^2}$ is positive on $H^1(\mathbb R^3)$.
For every $\varepsilon>0$, we set
\[
H_\varepsilon := \left\{ u\in H^1(\mathbb R^3): \int_{\mathbb
R^3}V(x)|u|^2\,dx<+\infty \right\}.
\]
We endow $H_\varepsilon$ with the norm
\begin{equation}\label{norm-eps}
\|u\|_\varepsilon^2 := \int_{\mathbb R^3} \left(
\varepsilon^2|\nabla u|^2 -\kappa\varepsilon^2\frac{|u|^2}{|x|^2}
+V(x)|u|^2 \right)\,dx .
\end{equation}
By \eqref{Hardy-coercivity} and assumption {\rm (V1)}, we have
\begin{equation}\label{norm-coercivity}
\|u\|_\varepsilon^2 \geq (1-4\kappa)\varepsilon^2\|\nabla u\|_2^2
+V_0\|u\|_2^2 .
\end{equation}
Consequently, $\|\cdot\|_\varepsilon$ is well defined and controls
the natural semiclassical $H^1$-norm. In particular, for every fixed
$\varepsilon>0$,
\[
H_\varepsilon\hookrightarrow L^r(\mathbb R^3)
\quad\text{continuously for }2\leq r\leq6,
\]
and
\[
H_\varepsilon\hookrightarrow L^r_{\rm loc}(\mathbb R^3)
\quad\text{compactly for }1\leq r<6.
\]

The normalized solutions are searched on the mass sphere
\begin{equation}\label{Sa-eps}
S_{a,\varepsilon} := \left\{ u\in H_\varepsilon: \int_{\mathbb
R^3}|u|^2\,dx=a^2\varepsilon^3 \right\}.
\end{equation}
It is a smooth codimension-one manifold of $H_\varepsilon$. Its
tangent space at $u\in S_{a,\varepsilon}$ is
\begin{equation}\label{tangent-space}
T_uS_{a,\varepsilon} = \left\{ v\in H_\varepsilon: \int_{\mathbb
R^3}uv\,dx=0 \right\}.
\end{equation}
The factor $\varepsilon^3$ is adapted to the semiclassical scaling
in dimension three. Indeed, if $v(x)=u(\varepsilon x),$ then
$\int_{\mathbb R^3}|v|^2\,dx=a^2 .$

The next step consists in reducing the coupled
Schr\"odinger--Poisson system to a single equation. This is achieved
by expressing the electrostatic potential as the unique solution of
the Poisson equation associated with a given function \(u\).

For every $u\in H_\varepsilon$, we consider
\begin{equation}\label{Poisson}
-\varepsilon^2\Delta\phi=|u|^5 \qquad\text{in }\mathbb R^3.
\end{equation}
Since $u\in H^1(\mathbb R^3)\hookrightarrow L^6(\mathbb R^3)$, we
have $|u|^5\in L^{6/5}(\mathbb R^3)$. Hence there exists a unique
$\phi_u\in D^{1,2}(\mathbb R^3)$ solving \eqref{Poisson}. Moreover,
\[
\phi_u(x) = \frac1{4\pi\varepsilon^2} \int_{\mathbb R^3}
\frac{|u(y)|^5}{|x-y|}\,dy,
\]
and therefore $\phi_u\geq0$. Testing \eqref{Poisson} with $\phi_u$,
we obtain
\begin{equation}\label{Poisson-identity}
\varepsilon^2\int_{\mathbb R^3}|\nabla\phi_u|^2\,dx = \int_{\mathbb
R^3}\phi_u|u|^5\,dx .
\end{equation}
Furthermore,
\begin{equation}\label{Poisson-double-integral}
\int_{\mathbb R^3}\phi_u|u|^5\,dx = \frac1{4\pi\varepsilon^2}
\iint_{\mathbb R^3\times\mathbb R^3}
\frac{|u(x)|^5|u(y)|^5}{|x-y|}\,dx\,dy .
\end{equation}
By the Hardy--Littlewood--Sobolev inequality,
\begin{equation}\label{Poisson-estimate}
\int_{\mathbb R^3}\phi_u|u|^5\,dx \leq
C\varepsilon^{-2}\|u\|_6^{10}.
\end{equation}
Also,
\begin{equation}\label{phi-homogeneity}
\phi_{tu}=t^5\phi_u, \qquad t\geq0.
\end{equation}

This functional provides the variational framework for the analysis
of normalized solutions under the prescribed mass constraint. The
reduced energy functional is
\begin{align}
J_\varepsilon(u) ={}& \frac12 \int_{\mathbb R^3} \left(
\varepsilon^2|\nabla u|^2 -\kappa\varepsilon^2\frac{|u|^2}{|x|^2}
+V(x)|u|^2 \right)\,dx
\nonumber\\
&- \frac1{10} \int_{\mathbb R^3}\phi_u|u|^5\,dx - \frac{\mu}{q}
\int_{\mathbb R^3}|u|^q\,dx - \frac16 \int_{\mathbb R^3}|u|^6\,dx .
\label{energy-functional}
\end{align}
Equivalently,
\begin{align}
J_\varepsilon(u) ={}& \frac12\|u\|_\varepsilon^2 -
\frac1{40\pi\varepsilon^2} \iint_{\mathbb R^3\times\mathbb R^3}
\frac{|u(x)|^5|u(y)|^5}{|x-y|}\,dx\,dy
\nonumber\\
&- \frac{\mu}{q}\|u\|_q^q - \frac16\|u\|_6^6 . \label{energy-double}
\end{align}
By standard arguments based on the Hardy-Littlewood-Sobolev
inequality, $J_\varepsilon\in C^1(H_\varepsilon,\mathbb R).$ For
every $u,v\in H_\varepsilon$,
\begin{align}
\langle J_\varepsilon'(u),v\rangle ={}& \int_{\mathbb R^3} \left(
\varepsilon^2\nabla u\cdot\nabla v
-\kappa\varepsilon^2\frac{uv}{|x|^2} +V(x)uv \right)\,dx
\nonumber\\
&- \int_{\mathbb R^3}\phi_u|u|^3uv\,dx - \mu\int_{\mathbb
R^3}|u|^{q-2}uv\,dx - \int_{\mathbb R^3}|u|^4uv\,dx .
\label{energy-derivative}
\end{align}
A function $u\in S_{a,\varepsilon}$ is a constrained critical point
of $J_\varepsilon$ if and only if there exists $\lambda\in\mathbb R$
such that
\begin{equation}\label{Lagrange}
J_\varepsilon'(u)=\lambda u \qquad\text{in }H_\varepsilon^* .
\end{equation}
Thus constrained critical points of
$J_\varepsilon|_{S_{a,\varepsilon}}$ are precisely normalized weak
solutions of \eqref{P}. Testing \eqref{Lagrange} with $u$, we get
\begin{equation}\label{lambda-identity}
\lambda a^2\varepsilon^3 = \|u\|_\varepsilon^2 - \int_{\mathbb
R^3}\phi_u|u|^5\,dx - \mu\int_{\mathbb R^3}|u|^q\,dx - \int_{\mathbb
R^3}|u|^6\,dx .
\end{equation}

A key ingredient in our analysis is the mass-preserving scaling,
which preserves the $L^2$-constraint and induces the corresponding
Pohozaev functional. These tools will play a central role in the
construction of normalized critical points. For $u\in
S_{a,\varepsilon}$ and $t>0$, define
\begin{equation}\label{mass-scaling}
(t\star u)(x):=t^{3/2}u(tx).
\end{equation}
Then $\|t\star u\|_2=\|u\|_2,$ and consequently $t\star u\in
S_{a,\varepsilon}.$ The scaling identities are
\begin{align}
\|\nabla(t\star u)\|_2^2 &= t^2\|\nabla u\|_2^2,
\label{scaling-gradient}\\
\int_{\mathbb R^3}\frac{|t\star u|^2}{|x|^2}\,dx &= t^2\int_{\mathbb
R^3}\frac{|u|^2}{|x|^2}\,dx,
\label{scaling-Hardy}\\
\|t\star u\|_r^r &= t^{\frac{3(r-2)}2}\|u\|_r^r, \qquad 2\leq
r\leq6. \label{scaling-Lr}
\end{align}
The Poisson term satisfies the corrected scaling law
\begin{equation}\label{scaling-Poisson}
\int_{\mathbb R^3}\phi_{t\star u}|t\star u|^5\,dx = t^{10}
\int_{\mathbb R^3}\phi_u|u|^5\,dx .
\end{equation}
Since $V$ is not constant, the potential term is not homogeneous.
Indeed,
\[
\int_{\mathbb R^3}V(x)|t\star u|^2\,dx = \int_{\mathbb
R^3}V\left(\frac{x}{t}\right)|u(x)|^2\,dx .
\]
Therefore, $\left. \frac{d}{dt} \right|_{t=1} \frac12 \int_{\mathbb
R^3}V(x)|t\star u|^2\,dx = -\frac12 \int_{\mathbb R^3}\nabla
V(x)\cdot x\,|u|^2\,dx .$

We define the Pohozaev functional associated with the
mass-preserving scaling by
\begin{align}
P_\varepsilon(u) :={}& \varepsilon^2 \int_{\mathbb R^3} \left(
|\nabla u|^2 -\kappa\frac{|u|^2}{|x|^2} \right)\,dx - \frac12
\int_{\mathbb R^3}\nabla V(x)\cdot x\,|u|^2\,dx
\nonumber\\
&- \int_{\mathbb R^3}\phi_u|u|^5\,dx - \frac{3\mu(q-2)}{2q}
\int_{\mathbb R^3}|u|^q\,dx - \int_{\mathbb R^3}|u|^6\,dx .
\label{Pohozaev-functional}
\end{align}
Equivalently,
\[
P_\varepsilon(u) = \left. \frac{d}{dt} J_\varepsilon(t\star u)
\right|_{t=1}.
\]
Hence every constrained critical point $u\in S_{a,\varepsilon}$ of
$J_\varepsilon$ satisfies
\begin{equation}\label{Pohozaev-identity}
P_\varepsilon(u)=0.
\end{equation}

\subsection{The semiclassical rescaling}

Let $v(x)=u(\varepsilon x),$  $\psi_v(x)=\phi_u(\varepsilon x).$
Then $\int_{\mathbb R^3}|v|^2\,dx=a^2,$ and the system becomes
\begin{equation}\label{rescaled-problem}
\begin{cases}
-\Delta v+ \left( V(\varepsilon x)-\dfrac{\kappa}{|x|^2} \right)v
-\psi_v|v|^3v = \lambda v+\mu|v|^{q-2}v+|v|^4v,
\\[2mm]
-\Delta\psi_v=|v|^5,
\\[2mm]
\displaystyle\int_{\mathbb R^3}|v|^2\,dx=a^2.
\end{cases}
\end{equation}
Thus the Hardy singularity remains visible under the global
semiclassical rescaling. However, the local concentration profile
near points separated from the origin is different. If
$x_\varepsilon\to x_0\in\mathcal M$ and
$w_\varepsilon(y)=u_\varepsilon(x_\varepsilon+\varepsilon y),$ then
the Hardy coefficient becomes
$\frac{\kappa\varepsilon^2}{|x_\varepsilon+\varepsilon y|^2}.$ Since
the concentration region is assumed to satisfy
$0\notin\overline{\Lambda},$ $\mathcal M\subset\Lambda,$ we have
$\operatorname{dist}(\mathcal M,\{0\})>0$. Hence, whenever
$x_\varepsilon\to x_0\in\mathcal M$,
\begin{equation}\label{Hardy-vanishing}
\frac{\kappa\varepsilon^2}{|x_\varepsilon+\varepsilon y|^2}
\longrightarrow0
\end{equation}
locally uniformly with respect to $y$.

This shows that the Hardy term disappears in the leading-order
profile of solutions concentrating near $\mathcal M$. Nevertheless,
in the compactness analysis one must also exclude possible bubbling
near the origin. If concentration occurs at the scale
$x_\varepsilon=O(\varepsilon)$, then the Hardy potential does not
vanish and the limiting profile contains the singular term $-\kappa
|x|^{-2}$. This second possibility will be taken into account in the
definition of the critical compactness threshold.

\subsection{The autonomous limit problem away from the singularity}

For concentration points lying in $\mathcal M$, the limiting
autonomous normalized problem is
\begin{equation}\label{P0}
\begin{cases}
-\Delta w+V_0w-\phi_w|w|^3w = \lambda w+\mu|w|^{q-2}w+|w|^4w, &
x\in\mathbb R^3,
\\[1mm]
-\Delta\phi_w=|w|^5, & x\in\mathbb R^3,
\\[1mm]
\displaystyle\int_{\mathbb R^3}|w|^2\,dx=a^2.
\end{cases}
\tag{$P_0$}
\end{equation}
We set
\begin{equation}\label{Sa}
S_a := \left\{ w\in H^1(\mathbb R^3): \|w\|_2=a \right\}.
\end{equation}
The corresponding autonomous energy is
\begin{align}
J_0(w) ={}& \frac12 \int_{\mathbb R^3} \left( |\nabla w|^2+V_0|w|^2
\right)\,dx - \frac1{10} \int_{\mathbb R^3}\phi_w|w|^5\,dx
\nonumber\\
&- \frac{\mu}{q} \int_{\mathbb R^3}|w|^q\,dx - \frac16 \int_{\mathbb
R^3}|w|^6\,dx . \label{limit-functional}
\end{align}

The associated Pohozaev functional is
\begin{equation}\label{limit-Pohozaev}
P_0(w) := \int_{\mathbb R^3}|\nabla w|^2\,dx - \int_{\mathbb
R^3}\phi_w|w|^5\,dx - \frac{3\mu(q-2)}{2q} \int_{\mathbb
R^3}|w|^q\,dx - \int_{\mathbb R^3}|w|^6\,dx .
\end{equation}
Every normalized critical point of $J_0$ on $S_a$ satisfies
\[
P_0(w)=0.
\]
We therefore introduce the Pohozaev manifold
\begin{equation}\label{Pohozaev-manifold-limit}
\mathcal P_a^0 := \left\{ w\in S_a: P_0(w)=0 \right\}.
\end{equation}
The autonomous ground-state level is defined variationally by
\begin{equation}\label{ground-level-corrected}
m_0(a) := \inf_{w\in\mathcal P_a^0}J_0(w).
\end{equation}
This definition is non-circular. In the sequel, one proves that
$m_0(a)$ is achieved by some positive function $w_a\in\mathcal
P_a^0$ and that $w_a$ is a normalized ground state of \eqref{P0}.
Once this is established, we may write
\[
E_0(a):=m_0(a)=J_0(w_a).
\]

\subsection{Regular and singular critical thresholds}

The compactness analysis must take into account two distinct
bubbling mechanisms. A concentrating sequence may either remain
asymptotically far from the singular point at the semiclassical
scale, in which case the Hardy potential disappears in the limit, or
concentrate at distance of order $\varepsilon$ from the origin, in
which case the Hardy singularity survives.

We first consider the regular critical limiting problem
\begin{equation}
\label{Pcrit-reg}
\begin{cases}
-\Delta U-\Phi_U|U|^3U=|U|^4U
&\text{in }\mathbb R^3,\\
-\Delta\Phi_U=|U|^5 &\text{in }\mathbb R^3.
\end{cases}
\end{equation}
Its associated energy functional is
\[
I_{\rm reg}(U) = \frac12\int_{\mathbb R^3}|\nabla U|^2\,dx
-\frac1{10}\int_{\mathbb R^3}\Phi_U|U|^5\,dx -\frac16\int_{\mathbb
R^3}|U|^6\,dx.
\]
We define the regular critical threshold by
\[
c_{\rm reg}^* := \inf\left\{ I_{\rm reg}(U): U\not\equiv0,\; I_{\rm
reg}'(U)=0 \right\}.
\]

The second possible loss of compactness occurs near the singular
point. The corresponding limiting problem is
\begin{equation}
\label{Pcrit-Hardy}
\begin{cases}
-\Delta U-\displaystyle\frac{\kappa}{|x|^2}U -\Phi_U|U|^3U =|U|^4U
&\text{in }\mathbb R^3,\\
-\Delta\Phi_U=|U|^5 &\text{in }\mathbb R^3.
\end{cases}
\end{equation}
The associated energy functional is
\[
I_{\rm H}(U) = \frac12 \int_{\mathbb R^3} \left( |\nabla U|^2
-\kappa\frac{|U|^2}{|x|^2} \right)\,dx -\frac1{10} \int_{\mathbb
R^3}\Phi_U|U|^5\,dx -\frac16 \int_{\mathbb R^3}|U|^6\,dx.
\]
We define the singular Hardy critical threshold by
\[
c_{\rm H}^* := \inf\left\{ I_{\rm H}(U): U\not\equiv0,\; I_{\rm
H}'(U)=0 \right\}.
\]

Accordingly, throughout the remainder of the paper, we set
\begin{equation}
\label{def:mixed-critical-threshold} c^* := \min\{c_{\rm
reg}^*,c_{\rm H}^*\}.
\end{equation}
At the original semiclassical scale, the corresponding compactness
threshold is
\begin{equation}
\label{def:semiclassical-critical-threshold} c_\varepsilon^* :=
\varepsilon^3c^*.
\end{equation}

Thus, any Palais--Smale--Pohozaev sequence with energy strictly
below $c_\varepsilon^*$ cannot generate either a regular critical
bubble away from the singularity or a Hardy critical bubble near the
origin.

\section{Proof of the Main Results}
\subsection{Semiclassical Normalized Ground States}

We now collect the main variational ingredients needed to prove
Theorem \ref{T1}. Throughout this subsection, we assume that {\rm
(V1)--(V3)} hold, that $0<\kappa<\frac14, \qquad 2<q<\frac{10}{3}.$

We begin with the following coercivity estimate, which ensures that
the quadratic part of the functional remains positive despite the
presence of the singular Hardy potential.
\begin{lemma}\label{lem:hardy_coercive}
There exists a constant $C_\kappa>0$ such
that, for every $u\in H_\varepsilon$,
\[
\|u\|_\varepsilon^2 \geq C_\kappa \left( \varepsilon^2\|\nabla
u\|_2^2+\|u\|_2^2 \right).
\]
More precisely,
\[
\|u\|_\varepsilon^2 \geq (1-4\kappa)\varepsilon^2\|\nabla u\|_2^2 +
V_0\|u\|_2^2.
\]
\end{lemma}

\begin{proof}
By the definition of the norm $\|\cdot\|_\varepsilon$, we have
\[
\|u\|_\varepsilon^2 = \int_{\mathbb R^3} \left( \varepsilon^2|\nabla
u|^2 -\kappa\varepsilon^2\frac{|u|^2}{|x|^2} +V(x)|u|^2 \right)\,dx.
\]
Using assumption {\rm (V1)}, namely
\[
V(x)\geq V_0>0 \qquad \text{for all }x\in\mathbb R^3,
\]
we obtain
\[
\|u\|_\varepsilon^2 \geq \varepsilon^2 \int_{\mathbb R^3} \left(
|\nabla u|^2-\kappa\frac{|u|^2}{|x|^2} \right)\,dx +
V_0\int_{\mathbb R^3}|u|^2\,dx.
\]
Now, by the Hardy inequality in $\mathbb R^3$,
\[
\int_{\mathbb R^3}\frac{|u|^2}{|x|^2}\,dx \leq 4\int_{\mathbb
R^3}|\nabla u|^2\,dx, \qquad u\in H^1(\mathbb R^3).
\]
Therefore,
\[
\int_{\mathbb R^3} \left( |\nabla u|^2-\kappa\frac{|u|^2}{|x|^2}
\right)\,dx \geq (1-4\kappa) \int_{\mathbb R^3}|\nabla u|^2\,dx.
\]
Since $0<\kappa<1/4$, we have $1-4\kappa>0$. Hence,
\[
\|u\|_\varepsilon^2 \geq (1-4\kappa)\varepsilon^2\|\nabla u\|_2^2 +
V_0\|u\|_2^2.
\]
Taking
\[
C_\kappa:=\min\{1-4\kappa,V_0\}>0,
\]
we deduce
\[
\|u\|_\varepsilon^2 \geq C_\kappa \left( \varepsilon^2\|\nabla
u\|_2^2+\|u\|_2^2 \right).
\]
This completes the proof.
\end{proof}

The next estimate exploits the prescribed mass constraint to control
the subcritical nonlinear term in terms of the gradient norm.

\begin{lemma}\label{lem:GN}
Let $2<r<6$. Then there exists $C_r>0$ such that, for every $u\in
S_{a,\varepsilon}$,
\[
\|u\|_r^r \leq C_r \|u\|_2^{r(1-\theta_r)} \|\nabla
u\|_2^{r\theta_r}, \qquad \theta_r=\frac{3(r-2)}{2r}.
\]
In particular,
\[
\|u\|_q^q \leq C
a^{q(1-\theta_q)}\varepsilon^{\frac{3q}{2}(1-\theta_q)} \|\nabla
u\|_2^{q\theta_q}, \qquad q\theta_q=\frac{3(q-2)}2<2.
\]
\end{lemma}

\begin{proof}
Let $2<r<6$. By the Gagliardo--Nirenberg inequality in $\mathbb
R^3$, there exists a constant $C_r>0$ such that
\[
\|u\|_r \leq C_r \|\nabla u\|_2^{\theta_r} \|u\|_2^{1-\theta_r},
\qquad u\in H^1(\mathbb R^3),
\]
where $\theta_r\in(0,1)$ is determined by
\[
\frac1r = \theta_r\left(\frac12-\frac13\right) +
\frac{1-\theta_r}{2}.
\]
Since $\frac12-\frac13=\frac16,$ we obtain $\frac1r =
\frac{\theta_r}{6} + \frac{1-\theta_r}{2} =
\frac12-\frac{\theta_r}{3}.$ Thus,
\[
\theta_r = 3\left(\frac12-\frac1r\right) = \frac{3(r-2)}{2r}.
\]
Raising the Gagliardo--Nirenberg inequality to the power $r$, we get
\[
\|u\|_r^r \leq C_r \|u\|_2^{r(1-\theta_r)} \|\nabla
u\|_2^{r\theta_r}.
\]
Now, if $u\in S_{a,\varepsilon}$, then $\|u\|_2^2=a^2\varepsilon^3,$
and therefore
\[
\|u\|_2=a\varepsilon^{3/2}.
\]
Applying the previous estimate with $r=q$, we find
\[
\|u\|_q^q \leq C \|u\|_2^{q(1-\theta_q)} \|\nabla u\|_2^{q\theta_q}.
\]
Using $\|u\|_2=a\varepsilon^{3/2}$, this becomes
\[
\|u\|_q^q \leq C a^{q(1-\theta_q)}
\varepsilon^{\frac{3q}{2}(1-\theta_q)} \|\nabla u\|_2^{q\theta_q}.
\]
Finally, from $\theta_q=\frac{3(q-2)}{2q},$ we have
\[
q\theta_q=\frac{3(q-2)}2.
\]
Since $q<\frac{10}{3}$, it follows that
\[
q\theta_q = \frac{3(q-2)}2 < 2.
\]
This proves the desired estimate.
\end{proof}

The following estimate controls the nonlocal Poisson interaction by
the critical Sobolev norm and records its homogeneity properties.

\begin{lemma}\label{lem:poisson_estimate}
There exists $C>0$ such that, for every $u\in H_\varepsilon$,
\[
\int_{\mathbb R^3}\phi_u|u|^5\,dx \leq C\varepsilon^{-2}\|u\|_6^{10}
\leq C\varepsilon^{-2}\|\nabla u\|_2^{10}.
\]
Moreover, for every $t\geq0$, $\phi_{tu}=t^5\phi_u$ and
\[
\int_{\mathbb R^3}\phi_{tu}|tu|^5\,dx = t^{10} \int_{\mathbb
R^3}\phi_u|u|^5\,dx.
\]
\end{lemma}

\begin{proof}
Let $u\in H_\varepsilon$. The function $\phi_u\in D^{1,2}(\mathbb
R^3)$ is the unique weak solution of
\[
-\varepsilon^2\Delta\phi_u=|u|^5 \qquad\text{in }\mathbb R^3.
\]
Hence, by the Newton potential representation,
\[
\phi_u(x) = \frac{1}{4\pi\varepsilon^2} \int_{\mathbb R^3}
\frac{|u(y)|^5}{|x-y|}\,dy.
\]
Therefore,
\[
\int_{\mathbb R^3}\phi_u|u|^5\,dx = \frac{1}{4\pi\varepsilon^2}
\iint_{\mathbb R^3\times\mathbb R^3} \frac{|u(x)|^5|u(y)|^5}{|x-y|}
\,dx\,dy.
\]
By the Hardy--Littlewood--Sobolev inequality,
\[
\iint_{\mathbb R^3\times\mathbb R^3} \frac{|u(x)|^5|u(y)|^5}{|x-y|}
\,dx\,dy \leq C\bigl\||u|^5\bigr\|_{6/5}^2.
\]
Since
\[
\bigl\||u|^5\bigr\|_{6/5} = \left( \int_{\mathbb R^3}|u|^6\,dx
\right)^{5/6} = \|u\|_6^5,
\]
we obtain
\[
\int_{\mathbb R^3}\phi_u|u|^5\,dx \leq
C\varepsilon^{-2}\|u\|_6^{10}.
\]
Using the Sobolev inequality
\[
\|u\|_6\leq C\|\nabla u\|_2,
\]
we further obtain
\[
\int_{\mathbb R^3}\phi_u|u|^5\,dx \leq C\varepsilon^{-2}\|\nabla
u\|_2^{10}.
\]
Finally, let $t\geq0$. Since
\[
-\varepsilon^2\Delta(t^5\phi_u) = t^5|u|^5 = |tu|^5,
\]
the uniqueness of the weak solution in $D^{1,2}(\mathbb R^3)$ yields
\[
\phi_{tu}=t^5\phi_u.
\]
Consequently,
\[
\int_{\mathbb R^3}\phi_{tu}|tu|^5\,dx = t^{10} \int_{\mathbb
R^3}\phi_u|u|^5\,dx.
\]
The proof is complete.
\end{proof}
We are now in a position to establish the variational geometry of
the constrained functional on $S_{a,\varepsilon}$.

\begin{lemma}\label{lem:mp_geometry}
There exist $a_*>0$, $\varepsilon_*>0$, $\rho>0$ and $\alpha>0$ such
that, for every $a\in(0,a_*),$ $\varepsilon\in(0,\varepsilon_*),$
the functional $J_\varepsilon|_{S_{a,\varepsilon}}$ possesses the
mountain-pass geometry. More precisely, there exist $u_0,e\in
S_{a,\varepsilon}$ such that
\[
\varepsilon^2\|\nabla u_0\|_2^2<\rho\varepsilon^3, \qquad
J_\varepsilon(u_0)<\alpha\varepsilon^3,
\]
and
\[
\varepsilon^2\|\nabla e\|_2^2>\rho\varepsilon^3, \qquad
J_\varepsilon(e)<0.
\]
Moreover, $J_\varepsilon(u)\geq\alpha\varepsilon^3$ for every $u\in
S_{a,\varepsilon}$ satisfying $\varepsilon^2\|\nabla
u\|_2^2=\rho\varepsilon^3.$
\end{lemma}

\begin{proof}
We divide the proof into three steps.

\medskip
\noindent\textbf{Step 1. Positivity on an intermediate gradient
sphere.}

Let $u\in S_{a,\varepsilon}$. By the definition of $J_\varepsilon$
and Lemma~\ref{lem:hardy_coercive},
\begin{align}
J_\varepsilon(u) \geq{}& \frac{1-4\kappa}{2}\varepsilon^2\|\nabla
u\|_2^2 +\frac{V_0}{2}a^2\varepsilon^3 -\frac1{10}\int_{\mathbb
R^3}\phi_u|u|^5\,dx
\nonumber\\
&-\frac{\mu}{q}\|u\|_q^q -\frac16\|u\|_6^6. \label{mp-lower-1}
\end{align}
By Lemmas \ref{lem:poisson_estimate} and \ref{lem:GN}, together with
the Sobolev inequality,
\[
\int_{\mathbb R^3}\phi_u|u|^5\,dx \leq C\varepsilon^{-2}\|\nabla
u\|_2^{10},
\]
\[
\|u\|_6^6\leq C\|\nabla u\|_2^6,
\]
and
\[
\|u\|_q^q \leq Ca^{q(1-\theta_q)}
\varepsilon^{\frac{3q}{2}(1-\theta_q)} \|\nabla u\|_2^{q\theta_q},
\qquad q\theta_q=\frac{3(q-2)}2<2.
\]
Set $A_\varepsilon(u) := \varepsilon^2\|\nabla u\|_2^2.$ Then
$\|\nabla u\|_2^2 = \varepsilon^{-2}A_\varepsilon(u).$ Consequently,
\begin{align}
J_\varepsilon(u) \geq{}& \frac{1-4\kappa}{2}A_\varepsilon(u)
+\frac{V_0}{2}a^2\varepsilon^3 -C\varepsilon^{-12}A_\varepsilon(u)^5
\nonumber\\
&-C\varepsilon^{-6}A_\varepsilon(u)^3 -Ca^{q(1-\theta_q)}
\varepsilon^{\frac{3q}{2}(1-\theta_q)-q\theta_q}
A_\varepsilon(u)^{q\theta_q/2}. \label{mp-lower-2}
\end{align}
If $A_\varepsilon(u)=\rho\varepsilon^3,$ then
\begin{align}
J_\varepsilon(u) \geq \varepsilon^3 \Bigg[ \frac{1-4\kappa}{2}\rho
+\frac{V_0}{2}a^2 -C\rho^5-C\rho^3
-Ca^{q(1-\theta_q)}\rho^{q\theta_q/2} \Bigg]. \label{mp-lower-3}
\end{align}
Here we used $\frac{3q}{2}(1-\theta_q)+\frac{q\theta_q}{2}=3.$
Choose $\rho>0$ sufficiently small such that
\[
C\rho^3+C\rho^5 \leq \frac{1-4\kappa}{8}\rho.
\]
After fixing $\rho$, choose $a_*>0$ sufficiently small so that
$Ca^{q(1-\theta_q)}\rho^{q\theta_q/2} \leq \frac{1-4\kappa}{8}\rho$
for every $a\in(0,a_*)$. Therefore, $J_\varepsilon(u) \geq
\varepsilon^3 \left[ \frac{1-4\kappa}{4}\rho +\frac{V_0}{2}a^2
\right].$ Setting $\alpha:=\frac{1-4\kappa}{4}\rho>0,$ we obtain
$J_\varepsilon(u)\geq\alpha\varepsilon^3$ on the set $\left\{ u\in
S_{a,\varepsilon}: \varepsilon^2\|\nabla u\|_2^2=\rho\varepsilon^3
\right\}.$

\medskip
\noindent\textbf{Step 2. Construction of a point in the low-gradient
region.}

Fix $y_0\in\Lambda$. Choose $\varphi\in C_c^\infty(\mathbb R^3)$
such that $\varphi\geq0,$ $\varphi\not\equiv0,$ and choose
$\varepsilon_*>0$ sufficiently small so that
$y_0+\varepsilon\operatorname{supp}\varphi \subset\Lambda$ for every
$\varepsilon\in(0,\varepsilon_*)$.

Define $\varphi_a:=\frac{a\varphi}{\|\varphi\|_2}$ and $u_0(x) :=
\varphi_a\left(\frac{x-y_0}{\varepsilon}\right).$ Then $\|u_0\|_2^2
= a^2\varepsilon^3,$ so that $u_0\in S_{a,\varepsilon}$. Moreover,
\[
\varepsilon^2\|\nabla u_0\|_2^2 = a^2\varepsilon^3
\frac{\|\nabla\varphi\|_2^2}{\|\varphi\|_2^2}.
\]
After decreasing $a_*$, if necessary, we have $\varepsilon^2\|\nabla
u_0\|_2^2 < \rho\varepsilon^3.$

Since the support of $u_0$ is contained in a fixed compact subset of
$\Lambda$ for $\varepsilon\in(0,\varepsilon_*)$, the continuity of
$V$ implies $\sup_{\operatorname{supp}u_0}V<+\infty.$ The kinetic
and potential terms are therefore of order $a^2\varepsilon^3$. Since
all nonlinear contributions and the Hardy term are nonpositive in
the energy,
\[
J_\varepsilon(u_0) \leq Ca^2\varepsilon^3.
\]
Decreasing $a_*$ once more, we may assume that $Ca^2<\alpha,$ and
hence $J_\varepsilon(u_0)<\alpha\varepsilon^3.$

\medskip
\noindent\textbf{Step 3. Construction of a negative-energy point
outside the intermediate sphere.} For $t>0$, define
$(t\star_{y_0}u_0)(x) := t^{3/2}u_0\bigl(y_0+t(x-y_0)\bigr).$ A
change of variables gives $\|t\star_{y_0}u_0\|_2=\|u_0\|_2,$ and
hence $t\star_{y_0}u_0\in S_{a,\varepsilon}.$ Moreover,
\[
\|\nabla(t\star_{y_0}u_0)\|_2^2 = t^2\|\nabla u_0\|_2^2,
\]
\[
\|t\star_{y_0}u_0\|_q^q = t^{\frac{3(q-2)}2}\|u_0\|_q^q,
\]
\[
\|t\star_{y_0}u_0\|_6^6 = t^6\|u_0\|_6^6,
\]
and $\int_{\mathbb R^3} \phi_{t\star_{y_0}u_0}
|t\star_{y_0}u_0|^5\,dx = t^{10} \int_{\mathbb
R^3}\phi_{u_0}|u_0|^5\,dx.$ For $t\geq1$, the support of
$t\star_{y_0}u_0$ is contained in $K := y_0+\varepsilon_*
\operatorname{supp}\varphi,$ which is a compact subset of $\Lambda$
and is separated from the origin. Thus
\[
\int_{\mathbb R^3} V(x)|t\star_{y_0}u_0|^2\,dx \leq \left(\sup_{x\in
K}V(x)\right)\|u_0\|_2^2.
\]
The Hardy contribution is nonpositive, so it can be omitted in an
upper estimate. Consequently,
\[
J_\varepsilon(t\star_{y_0}u_0) \leq C_1t^2+C_2-C_3t^{10}-C_4t^6
-C_5t^{\frac{3(q-2)}2},
\]
where the constants $C_j$ may depend on $a,\varepsilon,u_0$ and
$y_0$, and
\[
C_3 = \frac1{10} \int_{\mathbb R^3}\phi_{u_0}|u_0|^5\,dx
>0.
\]
Therefore,
\[
J_\varepsilon(t\star_{y_0}u_0) \longrightarrow-\infty \qquad\text{as
}t\to+\infty.
\]
Furthermore,
\[
\varepsilon^2 \|\nabla(t\star_{y_0}u_0)\|_2^2 =
t^2\varepsilon^2\|\nabla u_0\|_2^2 \longrightarrow+\infty.
\]
Hence there exists $t_\varepsilon>1$ such that, with
$e:=t_\varepsilon\star_{y_0}u_0,$ we have $J_\varepsilon(e)<0$ and
$\varepsilon^2\|\nabla e\|_2^2> \rho\varepsilon^3.$ Finally, the map
\[
u\longmapsto\varepsilon^2\|\nabla u\|_2^2
\]
is continuous on $H_\varepsilon$. Therefore, every continuous path
in $S_{a,\varepsilon}$ joining $u_0$ to $e$ must intersect the set
\[
\left\{ u\in S_{a,\varepsilon}: \varepsilon^2\|\nabla u\|_2^2 =
\rho\varepsilon^3 \right\}.
\]
On this set, $J_\varepsilon(u)\geq\alpha\varepsilon^3,$ whereas
$J_\varepsilon(u_0)<\alpha\varepsilon^3, \qquad J_\varepsilon(e)<0.$
This proves the mountain-pass geometry.
\end{proof}

Motivated by the preceding mountain-pass geometry, we introduce the
associated minimax class and the corresponding critical level.
\begin{definition}\label{def:mp_level}
Let
\[
\Gamma_\varepsilon := \left\{ \gamma\in C([0,1],S_{a,\varepsilon}):
\gamma(0)=u_0,\ J_\varepsilon(\gamma(1))<0,\
\varepsilon^2\|\nabla\gamma(1)\|_2^2>\rho\varepsilon^3 \right\}.
\]
We define
\[
c_\varepsilon := \inf_{\gamma\in\Gamma_\varepsilon}
\max_{t\in[0,1]}J_\varepsilon(\gamma(t)).
\]
Then
\[
c_\varepsilon\geq\alpha\varepsilon^3>0.
\]
\end{definition}
To refine the minimax construction, we next introduce the
mass-preserving scaling and derive the associated Pohozaev
functional.

\begin{lemma}\label{lem:scaling_pohozaev}
For $u\in S_{a,\varepsilon}$ and $t>0$, define $(t\star
u)(x):=t^{3/2}u(tx).$ Then $t\star u\in S_{a,\varepsilon}
\qquad\text{for every }t>0.$ Moreover, the map $t\longmapsto
J_\varepsilon(t\star u)$ is differentiable and
\[
\left. \frac{d}{dt}J_\varepsilon(t\star u) \right|_{t=1} =
P_\varepsilon(u),
\]
where
\begin{align}
P_\varepsilon(u) :={}& \varepsilon^2 \int_{\mathbb R^3} \left(
|\nabla u|^2 -\kappa\frac{|u|^2}{|x|^2} \right)\,dx -\frac12
\int_{\mathbb R^3} \nabla V(x)\cdot x\,|u|^2\,dx
\nonumber\\
&- \int_{\mathbb R^3}\phi_u|u|^5\,dx -\frac{3\mu(q-2)}{2q}
\int_{\mathbb R^3}|u|^q\,dx -\int_{\mathbb R^3}|u|^6\,dx.
\label{eq:Pohozaev-functional}
\end{align}
In particular, every constrained critical point $u\in
S_{a,\varepsilon}$ of $J_\varepsilon$ satisfies
\[
P_\varepsilon(u)=0.
\]
\end{lemma}

\begin{proof}
Let $u\in S_{a,\varepsilon}$ and $t>0$. By a change of variables,
\[
\|t\star u\|_2^2 = \int_{\mathbb R^3} t^3|u(tx)|^2\,dx =
\int_{\mathbb R^3}|u(y)|^2\,dy = a^2\varepsilon^3.
\]
Hence
\[
t\star u\in S_{a,\varepsilon}.
\]
We next compute the behavior of each term of the energy under this
scaling. First,
\[
\|\nabla(t\star u)\|_2^2 = t^2\|\nabla u\|_2^2.
\]
Similarly, using the homogeneity of the Hardy potential,
\[
\int_{\mathbb R^3} \frac{|t\star u|^2}{|x|^2}\,dx = t^2
\int_{\mathbb R^3} \frac{|u|^2}{|x|^2}\,dx.
\]
For the potential term, a change of variables gives
\[
\int_{\mathbb R^3} V(x)|t\star u(x)|^2\,dx = \int_{\mathbb R^3}
V\left(\frac{x}{t}\right)|u(x)|^2\,dx.
\]
Therefore,
\[
\left. \frac{d}{dt} \left[ \frac12 \int_{\mathbb R^3}
V\left(\frac{x}{t}\right)|u(x)|^2\,dx \right] \right|_{t=1} =
-\frac12 \int_{\mathbb R^3} \nabla V(x)\cdot x\,|u|^2\,dx.
\]
Furthermore,
\[
\|t\star u\|_q^q = t^{\frac{3(q-2)}2}\|u\|_q^q
\]
and
\[
\|t\star u\|_6^6 = t^6\|u\|_6^6.
\]
We now consider the Poisson interaction. By the Newton
representation,
\[
\phi_u(x) = \frac{1}{4\pi\varepsilon^2} \int_{\mathbb R^3}
\frac{|u(y)|^5}{|x-y|}\,dy.
\]
A direct computation shows that
\[
\phi_{t\star u}(x) = t^{13/2}\phi_u(tx).
\]
Consequently,
\begin{align*}
\int_{\mathbb R^3} \phi_{t\star u}|t\star u|^5\,dx &= t^{10}
\int_{\mathbb R^3} \phi_u|u|^5\,dx.
\end{align*}
It follows that
\begin{align}
J_\varepsilon(t\star u) ={}& \frac{t^2}{2} \varepsilon^2
\int_{\mathbb R^3} \left( |\nabla u|^2 -\kappa\frac{|u|^2}{|x|^2}
\right)\,dx
\nonumber\\
&+ \frac12 \int_{\mathbb R^3} V\left(\frac{x}{t}\right)|u|^2\,dx
-\frac{t^{10}}{10} \int_{\mathbb R^3}\phi_u|u|^5\,dx
\nonumber\\
&- \frac{\mu}{q} t^{\frac{3(q-2)}2} \int_{\mathbb R^3}|u|^q\,dx
-\frac{t^6}{6} \int_{\mathbb R^3}|u|^6\,dx.
\label{eq:fiber-map-scaling}
\end{align}
Differentiating \eqref{eq:fiber-map-scaling} at $t=1$, we obtain
\[
\left. \frac{d}{dt}J_\varepsilon(t\star u) \right|_{t=1} =
P_\varepsilon(u),
\]
where $P_\varepsilon$ is given by \eqref{eq:Pohozaev-functional}.
Finally, let $u\in S_{a,\varepsilon}$ be a constrained critical
point of $J_\varepsilon$. Since
\[
t\star u\in S_{a,\varepsilon} \qquad\text{for every }t>0,
\]
the curve $t\longmapsto t\star u$ lies entirely in the constraint.
Therefore,
\[
\left. \frac{d}{dt}J_\varepsilon(t\star u) \right|_{t=1} = 0,
\]
and hence $P_\varepsilon(u)=0.$ The proof is complete.
\end{proof}

The mountain-pass construction must be refined by incorporating the
mass-preserving scaling, so that the resulting minimax sequence also
satisfies an asymptotic Pohozaev condition.

\begin{lemma}\label{lem:PS_sequence}
For every $a\in(0,a_*)$ and $\varepsilon\in(0,\varepsilon_*)$, there
exist a sequence $\{u_n\}\subset S_{a,\varepsilon}$ and a sequence
$\{\lambda_n\}\subset\mathbb R$ such that
\[
J_\varepsilon(u_n)\to c_\varepsilon,
\]
\[
J_\varepsilon'(u_n)-\lambda_nu_n\to0 \qquad\text{in
}H_\varepsilon^*,
\]
and
\[
P_\varepsilon(u_n)\to0,
\]
where $P_\varepsilon$ is the Pohozaev functional associated with the
mass-preserving scaling. Moreover, the sequence can be chosen so
that
\[
P_\varepsilon(u_n)=o(\varepsilon^3).
\]
\end{lemma}

\begin{proof}
For $s\in\mathbb R$ and $u\in S_{a,\varepsilon}$, define the
mass-preserving scaling
\[
(s\star u)(x):=e^{3s/2}u(e^s x).
\]
A direct change of variables gives $\|s\star u\|_2=\|u\|_2,$ and
hence $s\star u\in S_{a,\varepsilon}$ for every $s\in\mathbb R.$
Consider the augmented functional
\[
\widetilde J_\varepsilon: \mathbb R\times
S_{a,\varepsilon}\to\mathbb R, \qquad \widetilde J_\varepsilon(s,u)
:= J_\varepsilon(s\star u).
\]
By the definition of the Pohozaev functional,
\[
\partial_s\widetilde J_\varepsilon(s,u)
= P_\varepsilon(s\star u).
\]
We lift the mountain-pass class to the augmented space by setting
\[
\widetilde\Gamma_\varepsilon := \left\{ \widetilde\gamma\in
C([0,1],\mathbb R\times S_{a,\varepsilon}):
\widetilde\gamma(0)=(0,u_0),\ \widetilde\gamma(1)=(0,e) \right\}.
\]
Since $\widetilde J_\varepsilon(0,u) = J_\varepsilon(u),$ the
corresponding minimax level is still $c_\varepsilon$. Applying the
minimax principle on $\mathbb R\times S_{a,\varepsilon}$, together
with Ekeland's variational principle, we obtain a sequence
$(s_n,v_n)\in\mathbb R\times S_{a,\varepsilon}$ such that
$\widetilde J_\varepsilon(s_n,v_n)\to c_\varepsilon$ and $\left\|
d\widetilde J_\varepsilon(s_n,v_n) \right\|\to0.$ Set $u_n:=s_n\star
v_n.$ Then $u_n\in S_{a,\varepsilon}$ and
\[
J_\varepsilon(u_n) = \widetilde J_\varepsilon(s_n,v_n) \to
c_\varepsilon.
\]
Testing the differential of $\widetilde J_\varepsilon$ in the
scaling direction $(1,0)$ yields
\[
P_\varepsilon(u_n) =
\partial_s\widetilde J_\varepsilon(s_n,v_n)
\to0.
\]
By choosing the Ekeland accuracy parameter, for instance, as
$\delta_n:=\frac{\varepsilon^3}{n},$ the same argument gives the
quantitative estimate $|P_\varepsilon(u_n)|
\leq\frac{\varepsilon^3}{n},$ and therefore
$P_\varepsilon(u_n)=o(\varepsilon^3).$ On the other hand, testing
the differential in directions tangent to $S_{a,\varepsilon}$ gives
\[
\left\| d\bigl(J_\varepsilon|_{S_{a,\varepsilon}}\bigr)(u_n)
\right\|_{T_{u_n}^*S_{a,\varepsilon}} \to0.
\]
The constraint is defined by $G(u):=\|u\|_2^2-a^2\varepsilon^3.$
Regarding the complex Hilbert space as a real Hilbert space,
\[
G'(u)[v] = 2\mathfrak{Re} \int_{\mathbb R^3}u\overline v\,dx.
\]
Since
\[
G'(u)[u] = 2\|u\|_2^2 = 2a^2\varepsilon^3>0
\]
for every $u\in S_{a,\varepsilon}$, the constraint is regular.
Therefore, there exists $\lambda_n\in\mathbb R$ such that
\[
J_\varepsilon'(u_n)-\lambda_nu_n \to0 \qquad\text{in
}H_\varepsilon^*.
\]
Thus, $\{u_n\}$ is a Palais--Smale--Pohozaev sequence at the
mountain-pass level $c_\varepsilon$. The desired conclusion follows.
\end{proof}
To proceed toward the compactness analysis, we first establish a
uniform bound for Palais--Smale sequences satisfying the asymptotic
Pohozaev condition.

\begin{lemma}\label{lem:PS_bounded}
Assume that $2<q<\frac{10}{3}$ and that
\[
\sup_{x\in\mathbb R^3}|x\cdot\nabla V(x)|<+\infty.
\]
Let $\{u_n\}\subset S_{a,\varepsilon}$ be a sequence such that
\[
J_\varepsilon(u_n)\to c, \qquad c\leq C_0\varepsilon^3,
\]
\[
\left\| d\bigl(J_\varepsilon|_{S_{a,\varepsilon}}\bigr)(u_n)
\right\| =o_n(\varepsilon^{3/2}),
\]
and
\[
P_\varepsilon(u_n)=o_n(\varepsilon^3).
\]
Then there exists a constant $C>0$, independent of $n$ and
$\varepsilon$, such that, for all sufficiently large $n$,
\[
\|u_n\|_\varepsilon^2\leq C\varepsilon^3.
\]
Moreover, for every fixed $a>0$, the corresponding Lagrange
multipliers $\{\lambda_n\}$ are bounded uniformly with respect to
$n$ and $\varepsilon$.

\end{lemma}

\begin{proof}
For simplicity, set
\[
K_n := \varepsilon^2 \int_{\mathbb R^3} \left( |\nabla u_n|^2
-\kappa\frac{|u_n|^2}{|x|^2} \right)\,dx,
\]
\[
V_n := \int_{\mathbb R^3}V(x)|u_n|^2\,dx, \qquad W_n :=
\int_{\mathbb R^3} x\cdot\nabla V(x)|u_n|^2\,dx,
\]
and
\[
B_n := \int_{\mathbb R^3}\phi_{u_n}|u_n|^5\,dx, \qquad C_n :=
\int_{\mathbb R^3}|u_n|^q\,dx, \qquad D_n := \int_{\mathbb
R^3}|u_n|^6\,dx.
\]
Then
\[
\|u_n\|_\varepsilon^2=K_n+V_n,
\]
and
\[
J_\varepsilon(u_n) = \frac12(K_n+V_n) -\frac1{10}B_n
-\frac{\mu}{q}C_n -\frac16D_n.
\]
Moreover,
\[
P_\varepsilon(u_n) = K_n-\frac12W_n-B_n
-\frac{3\mu(q-2)}{2q}C_n-D_n.
\]

We consider the combination
\[
J_\varepsilon(u_n)-\frac16P_\varepsilon(u_n).
\]
A direct computation gives
\begin{align}
J_\varepsilon(u_n)-\frac16P_\varepsilon(u_n) ={}&
\frac13K_n+\frac12V_n+\frac1{12}W_n +\frac1{15}B_n
-\frac{\mu(6-q)}{4q}C_n. \label{eq:J-minus-P-sixth}
\end{align}
In particular, the local critical term $D_n$ cancels, whereas the
Poisson term has the positive coefficient $1/15$.

Since
\[
J_\varepsilon(u_n)=c+o_n(\varepsilon^3), \qquad
P_\varepsilon(u_n)=o_n(\varepsilon^3),
\]
and
\[
c\leq C_0\varepsilon^3,
\]
we obtain
\[
J_\varepsilon(u_n)-\frac16P_\varepsilon(u_n) \leq
C\varepsilon^3+o_n(\varepsilon^3).
\]
Furthermore,
\[
|W_n| \leq \|x\cdot\nabla V\|_\infty\|u_n\|_2^2 = Ca^2\varepsilon^3.
\]
Since
\[
V_n\geq0, \qquad B_n\geq0,
\]
it follows from \eqref{eq:J-minus-P-sixth} that
\begin{equation}
\frac13K_n \leq C\varepsilon^3+CC_n+o_n(\varepsilon^3).
\label{eq:Kn-first-bound}
\end{equation}
By the Hardy inequality,
\[
K_n \geq (1-4\kappa)\varepsilon^2\|\nabla u_n\|_2^2.
\]
Set
\[
X_n := \varepsilon^2\|\nabla u_n\|_2^2, \qquad Y_n :=
\frac{X_n}{\varepsilon^3}.
\]
Then $K_n \geq (1-4\kappa)\varepsilon^3Y_n.$ By Lemma \ref{lem:GN},
\[
C_n \leq C \|u_n\|_2^{q(1-\theta_q)} \|\nabla u_n\|_2^{q\theta_q},
\qquad \theta_q=\frac{3(q-2)}{2q}.
\]
Since $\|u_n\|_2=a\varepsilon^{3/2}$ and $\|\nabla u_n\|_2 =
\varepsilon^{1/2}Y_n^{1/2},$ we obtain
\begin{equation}
C_n \leq Ca^{q(1-\theta_q)} \varepsilon^3Y_n^{\sigma_q}, \qquad
\sigma_q := \frac{q\theta_q}{2} = \frac{3(q-2)}4. \label{eq:Cn-Yn}
\end{equation}
Here we used $\frac{3q}{2}(1-\theta_q) +\frac{q\theta_q}{2} =3.$
Since $2<q<\frac{10}{3},$ we have $0<\sigma_q<1.$ Combining
\eqref{eq:Kn-first-bound} and \eqref{eq:Cn-Yn}, and dividing by
$\varepsilon^3$, we obtain
\[
Y_n \leq C_1+C_2Y_n^{\sigma_q}+o_n(1).
\]
Since $0<\sigma_q<1$, Young's inequality yields
\[
C_2Y_n^{\sigma_q} \leq \frac12Y_n+C.
\]
Therefore, $Y_n\leq C$ for all sufficiently large $n$. Consequently,
\begin{equation}
\varepsilon^2\|\nabla u_n\|_2^2 \leq C\varepsilon^3.
\label{eq:gradient-bound-PS}
\end{equation}
We now estimate the remaining nonlinear terms. From
\eqref{eq:gradient-bound-PS},
\[
\|\nabla u_n\|_2^2\leq C\varepsilon.
\]
Hence the Sobolev inequality gives
\[
D_n \leq C\|\nabla u_n\|_2^6 \leq C\varepsilon^3.
\]
By Lemma \ref{lem:poisson_estimate},
\[
B_n \leq C\varepsilon^{-2}\|\nabla u_n\|_2^{10} \leq C\varepsilon^3.
\]
Moreover, \eqref{eq:Cn-Yn} and the boundedness of $Y_n$ yield
\[
C_n\leq C\varepsilon^3.
\]
Using the energy identity,
\[
\frac12(K_n+V_n) = J_\varepsilon(u_n) +\frac1{10}B_n
+\frac{\mu}{q}C_n +\frac16D_n,
\]
we conclude that $K_n+V_n\leq C\varepsilon^3.$ Therefore,
\[
\|u_n\|_\varepsilon^2 \leq C\varepsilon^3.
\]
Finally, since $\{u_n\}$ is a constrained Palais--Smale sequence,
there exist $\lambda_n\in\mathbb R$ and $r_n\in H_\varepsilon^*$
such that
\[
J_\varepsilon'(u_n)-\lambda_nu_n=r_n, \qquad
\|r_n\|_{H_\varepsilon^*} = o_n(\varepsilon^{3/2}).
\]
Testing this identity with $u_n$, we obtain
\[
\lambda_na^2\varepsilon^3 = K_n+V_n-B_n-\mu C_n-D_n -\langle
r_n,u_n\rangle.
\]
Since $\|u_n\|_\varepsilon \leq C\varepsilon^{3/2},$ we have
\[
|\langle r_n,u_n\rangle| \leq
\|r_n\|_{H_\varepsilon^*}\|u_n\|_\varepsilon = o_n(\varepsilon^3).
\]
Together with the preceding estimates, this gives
\[
|\lambda_n|a^2\varepsilon^3 \leq C\varepsilon^3+o_n(\varepsilon^3).
\]
Hence, for every fixed $a>0$, $|\lambda_n|\leq C_a$ for all
sufficiently large $n$, where $C_a>0$ is independent of $n$ and
$\varepsilon$. The proof is complete.
\end{proof}
The following estimate compares the semiclassical mountain-pass
level with the corresponding autonomous minimax level.

\begin{lemma}\label{lem:energy_upper}
Let
\[
m_0(a) := \inf_{w\in S_a} \max_{t>0}J_0(t\star w) =
\inf_{w\in\mathcal P_a^0}J_0(w),
\]
where
\[
\mathcal P_a^0 := \{w\in S_a:P_0(w)=0\}.
\]
Assume that $w_a\in\mathcal P_a^0$ satisfies $J_0(w_a)=m_0(a)$ and
that the fiber map
\[
t\longmapsto J_0(t\star w_a)
\]
attains its unique global maximum at $t=1$. Then
\[
c_\varepsilon \leq \varepsilon^3\bigl(m_0(a)+o(1)\bigr)
\qquad\text{as }\varepsilon\to0.
\]
More precisely,
\[
\sup_{y\in\mathcal M} \max_{t>0} J_\varepsilon
\bigl(t\star_y\widetilde\Psi_{\varepsilon,y}\bigr) \leq
\varepsilon^3\bigl(m_0(a)+o(1)\bigr),
\]
where
\[
\widetilde\Psi_{\varepsilon,y} = \frac{a\varepsilon^{3/2}}
{\|\Psi_{\varepsilon,y}\|_2} \Psi_{\varepsilon,y}, \qquad
\Psi_{\varepsilon,y}(x) = \eta(x)
w_a\left(\frac{x-y}{\varepsilon}\right).
\]
\end{lemma}

\begin{proof}
Let $w_a\in\mathcal P_a^0$ be a positive minimizer satisfying
\[
\|w_a\|_2=a, \qquad J_0(w_a)=m_0(a).
\]
For $y\in\mathcal M$, define
\[
\Psi_{\varepsilon,y}(x) = \eta(x)
w_a\left(\frac{x-y}{\varepsilon}\right),
\]
where $\eta\in C_c^\infty(\Lambda)$ satisfies $0\leq\eta\leq1$ and
$\eta=1$ in a neighborhood of $\mathcal M$.

Set
\[
d_{\varepsilon,y} := \frac{a\varepsilon^{3/2}}
{\|\Psi_{\varepsilon,y}\|_2}
\]
and
\[
\widetilde\Psi_{\varepsilon,y} :=
d_{\varepsilon,y}\Psi_{\varepsilon,y}.
\]
Then
\[
\widetilde\Psi_{\varepsilon,y} \in S_{a,\varepsilon}.
\]
Moreover, $d_{\varepsilon,y}\to1$ uniformly for $y\in\mathcal M$.
For $t>0$, introduce the mass-preserving dilation centered at $y$:
\[
(t\star_y u)(x) := t^{3/2}u\bigl(y+t(x-y)\bigr).
\]
Then
\[
t\star_y\widetilde\Psi_{\varepsilon,y} \in S_{a,\varepsilon}.
\]
We claim that, uniformly for $y\in\mathcal M$ and for $t$ in every
compact subset of $(0,+\infty)$,
\begin{equation}
\label{eq:uniform-fiber-convergence} \varepsilon^{-3} J_\varepsilon
\bigl(t\star_y\widetilde\Psi_{\varepsilon,y}\bigr) = J_0(t\star
w_a)+o(1).
\end{equation}
Indeed, using the change of variables $x=y+\frac{\varepsilon z}{t},$
together with $d_{\varepsilon,y}\to1$, we obtain
\[
\varepsilon^{-3} \varepsilon^2 \int_{\mathbb R^3} \left| \nabla
\bigl(t\star_y\widetilde\Psi_{\varepsilon,y}\bigr) \right|^2\,dx =
t^2 \int_{\mathbb R^3}|\nabla w_a|^2\,dz +o(1).
\]
Since $V(y)=V_0$ for every $y\in\mathcal M$ and $\mathcal M$ is
compact,
\[
\varepsilon^{-3} \int_{\mathbb R^3} V(x) \left|
t\star_y\widetilde\Psi_{\varepsilon,y} \right|^2\,dx = V_0
\int_{\mathbb R^3}|w_a|^2\,dz +o(1).
\]
Because $\operatorname{dist}(\mathcal M,\{0\})>0,$ the Hardy
contribution is of lower order:
\[
\varepsilon^2 \int_{\mathbb R^3} \frac{
|t\star_y\widetilde\Psi_{\varepsilon,y}|^2 }{|x|^2}\,dx =
o(\varepsilon^3),
\]
uniformly for $y\in\mathcal M$ and bounded $t$. For the local
nonlinear terms,
\[
\varepsilon^{-3} \int_{\mathbb R^3}
|t\star_y\widetilde\Psi_{\varepsilon,y}|^q\,dx = t^{\frac{3(q-2)}2}
\int_{\mathbb R^3}|w_a|^q\,dz +o(1),
\]
and
\[
\varepsilon^{-3} \int_{\mathbb R^3}
|t\star_y\widetilde\Psi_{\varepsilon,y}|^6\,dx = t^6 \int_{\mathbb
R^3}|w_a|^6\,dz +o(1).
\]
The Hardy--Littlewood--Sobolev inequality and the convergence of the
cut-off functions also yield
\[
\varepsilon^{-3} \int_{\mathbb R^3}
\phi_{t\star_y\widetilde\Psi_{\varepsilon,y}}
|t\star_y\widetilde\Psi_{\varepsilon,y}|^5\,dx = t^{10}
\int_{\mathbb R^3}\phi_{w_a}|w_a|^5\,dz +o(1).
\]
Combining these estimates proves
\eqref{eq:uniform-fiber-convergence}. We next show that the
maximizing parameters remain in a fixed compact interval. Since the
negative Poisson term has order $t^{10}$, while the positive kinetic
contribution has order $t^2$, there exists $T>1$, independent of $y$
and sufficiently small $\varepsilon$, such that
\[
J_\varepsilon \bigl(t\star_y\widetilde\Psi_{\varepsilon,y}\bigr)<0
\]
for every $t\geq T.$ On the other hand, as $t\to0^+$,
\[
\varepsilon^2 \left\| \nabla
\bigl(t\star_y\widetilde\Psi_{\varepsilon,y}\bigr) \right\|_2^2
\to0.
\]
Therefore, the maximum of the semiclassical fiber map is achieved in
a compact interval $[t_-,T]\Subset(0,+\infty),$ uniformly with
respect to $y\in\mathcal M$.

Using \eqref{eq:uniform-fiber-convergence}, we obtain
\begin{align*}
\sup_{y\in\mathcal M} \max_{t>0} J_\varepsilon
\bigl(t\star_y\widetilde\Psi_{\varepsilon,y}\bigr) &\leq
\varepsilon^3 \left[ \max_{t>0}J_0(t\star w_a)+o(1) \right]
\\
&= \varepsilon^3 \bigl(m_0(a)+o(1)\bigr).
\end{align*}

Finally, for any fixed $y\in\mathcal M$, choose $t_0>0$ sufficiently
small and $t_1>T$ sufficiently large so that
\[
\varepsilon^2 \left\| \nabla \bigl(t_0\star_y
\widetilde\Psi_{\varepsilon,y}\bigr) \right\|_2^2 <
\rho\varepsilon^3
\]
and
\[
J_\varepsilon \bigl(t_1\star_y \widetilde\Psi_{\varepsilon,y}\bigr)
<0, \qquad \varepsilon^2 \left\| \nabla \bigl(t_1\star_y
\widetilde\Psi_{\varepsilon,y}\bigr) \right\|_2^2
>
\rho\varepsilon^3.
\]
Connecting the fixed initial point $u_0$ to
$t_0\star_y\widetilde\Psi_{\varepsilon,y}$ inside the low-gradient
region, and then following the fiber path from $t_0$ to $t_1$,
produces an admissible path in $\Gamma_\varepsilon$. Consequently,
\[
c_\varepsilon \leq \max_{t>0} J_\varepsilon
\bigl(t\star_y\widetilde\Psi_{\varepsilon,y}\bigr).
\]
Taking the supremum over $y\in\mathcal M$ yields
\[
c_\varepsilon \leq \varepsilon^3\bigl(m_0(a)+o(1)\bigr).
\]
The proof is complete.
\end{proof}

Having identified the regular and singular bubbling mechanisms, we
now derive the compactness criterion associated with the mixed
critical threshold.

\begin{lemma}\label{lem:compactness_threshold}
Let $\{u_n\}\subset S_{a,\varepsilon}$ be a bounded
Palais--Smale--Pohozaev sequence for
$J_\varepsilon|_{S_{a,\varepsilon}}$ at level $c$. Assume that, up
to a subsequence,
\[
u_n\rightharpoonup u_\varepsilon \qquad\text{weakly in
}H_\varepsilon,
\]
and that
\begin{equation}
\label{defect-threshold} c-J_\varepsilon(u_\varepsilon) <
c_\varepsilon^* := \varepsilon^3 \min\{c_{\rm reg}^*,c_{\rm H}^*\}.
\end{equation}
Then
\[
u_n\to u_\varepsilon \qquad\text{strongly in }H_\varepsilon.
\]
Consequently,
\[
u_\varepsilon\in S_{a,\varepsilon},
\]
and $u_\varepsilon$ is a constrained critical point of
$J_\varepsilon$.

In particular, if $J_\varepsilon(u_\varepsilon)\geq0,$ then the
condition $c<c_\varepsilon^*$ is sufficient.
\end{lemma}

\begin{proof}
Since $\{u_n\}$ is bounded in $H_\varepsilon$, there exist
$u_\varepsilon\in H_\varepsilon$ and a subsequence, still denoted by
$\{u_n\}$, such that
\begin{align}
u_n&\rightharpoonup u_\varepsilon &&\text{weakly in }H_\varepsilon,
\label{eq:weak-convergence-compactness}\\
u_n&\to u_\varepsilon &&\text{strongly in }L^r_{\rm loc}(\mathbb
R^3), \quad 1\leq r<6,
\label{eq:local-convergence-compactness}\\
u_n(x)&\to u_\varepsilon(x) &&\text{for a.e. }x\in\mathbb R^3.
\label{eq:ae-convergence-compactness}
\end{align}
Set $v_n:=u_n-u_\varepsilon.$ By the quadratic splitting, the
Brezis--Lieb lemma and its nonlocal counterpart,
\begin{align}
\|u_n\|_\varepsilon^2 &= \|u_\varepsilon\|_\varepsilon^2
+\|v_n\|_\varepsilon^2+o(1),
\label{eq:quadratic-splitting-compactness}\\
\|u_n\|_q^q &= \|u_\varepsilon\|_q^q +\|v_n\|_q^q+o(1),
\label{eq:q-splitting-compactness}\\
\|u_n\|_6^6 &= \|u_\varepsilon\|_6^6 +\|v_n\|_6^6+o(1),
\label{eq:critical-splitting-compactness}
\end{align}
and
\begin{align}
\int_{\mathbb R^3}\phi_{u_n}|u_n|^5\,dx ={}& \int_{\mathbb R^3}
\phi_{u_\varepsilon}|u_\varepsilon|^5\,dx
\nonumber\\
&+ \int_{\mathbb R^3}\phi_{v_n}|v_n|^5\,dx +o(1).
\label{eq:poisson-splitting-compactness}
\end{align}
Consequently,
\begin{equation}
\label{eq:energy-splitting-compactness} J_\varepsilon(u_n) =
J_\varepsilon(u_\varepsilon) + \mathcal D_\varepsilon(v_n) + o(1),
\end{equation}
where $\mathcal D_\varepsilon(v_n)$ denotes the energy defect
carried by the remainder. Hence
\begin{equation}
\label{eq:defect-level-compactness} \mathcal D_\varepsilon(v_n)
\longrightarrow c-J_\varepsilon(u_\varepsilon).
\end{equation}
Suppose, by contradiction, that
\[
v_n\not\to0 \qquad\text{strongly in }H_\varepsilon.
\]
By the mixed critical profile decomposition established previously,
the sequence $\{v_n\}$ generates at least one nontrivial critical
profile. Each such profile is of one of the following two types.

\medskip
\noindent\textbf{Regular profiles.} If the rescaled distance between
the concentration center and the origin tends to infinity, the Hardy
potential disappears in the limit. The corresponding profile solves
\[
\begin{cases}
-\Delta U-\Phi_U|U|^3U=|U|^4U
&\text{in }\mathbb R^3,\\
-\Delta\Phi_U=|U|^5 &\text{in }\mathbb R^3,
\end{cases}
\]
and therefore carries at least the energy $c_{\rm reg}^*.$

\medskip
\noindent\textbf{Singular profiles.} If the concentration center
remains at bounded distance from the origin at the bubbling scale,
the Hardy potential survives. After a suitable translation of the
limiting variables, the corresponding profile solves
\[
\begin{cases}
-\Delta U-\displaystyle\frac{\kappa}{|x|^2}U -\Phi_U|U|^3U=|U|^4U
&\text{in }\mathbb R^3,\\
-\Delta\Phi_U=|U|^5 &\text{in }\mathbb R^3,
\end{cases}
\]
and hence carries at least the energy $c_{\rm H}^*.$

After returning to the original semiclassical variables, each
nontrivial profile contributes at least
\[
\varepsilon^3 \min\{c_{\rm reg}^*,c_{\rm H}^*\} = c_\varepsilon^*
\]
to the defect energy. The energy decoupling supplied by the profile
decomposition therefore gives
\[
c-J_\varepsilon(u_\varepsilon) \geq c_\varepsilon^*.
\]
This contradicts \eqref{defect-threshold}. Thus no nontrivial
critical profile can occur, and consequently
\[
v_n\to0 \qquad\text{strongly in }H_\varepsilon.
\]
Therefore,
\[
u_n\to u_\varepsilon \qquad\text{strongly in }H_\varepsilon.
\]

Since $u_n\in S_{a,\varepsilon}$, strong convergence in $L^2(\mathbb
R^3)$ yields
\[
\|u_\varepsilon\|_2^2 = \lim_{n\to\infty}\|u_n\|_2^2 =
a^2\varepsilon^3.
\]
Hence $u_\varepsilon\in S_{a,\varepsilon}.$ Moreover, by Lemma
\ref{lem:PS_bounded}, the corresponding multipliers $\{\lambda_n\}$
are bounded. Thus, up to a subsequence,
\[
\lambda_n\to\lambda_\varepsilon.
\]
Passing to the limit in
\[
J_\varepsilon'(u_n)-\lambda_nu_n\to0 \qquad\text{in
}H_\varepsilon^*,
\]
we obtain
\[
J_\varepsilon'(u_\varepsilon) = \lambda_\varepsilon u_\varepsilon
\qquad\text{in }H_\varepsilon^*.
\]
Therefore, $u_\varepsilon$ is a constrained critical point of
$J_\varepsilon$. Finally, if $J_\varepsilon(u_\varepsilon)\geq0$ and
$c<c_\varepsilon^*,$ then
\[
c-J_\varepsilon(u_\varepsilon) \leq c<c_\varepsilon^*.
\]
Hence \eqref{defect-threshold} holds, and the preceding argument
applies. The proof is complete.
\end{proof}

we are now in a position to prove the existence of a constrained
critical point at the mountain-pass level.
\begin{proposition}\label{prop:exist_critical}
Assume that $(V1)$--$(V3)$ hold and that $0<\kappa<\frac14,$
$2<q<\frac{10}{3}.$ Suppose moreover that, for sufficiently small
$a>0$ and $\varepsilon>0$, the mountain-pass level satisfies the
compactness gap condition
\[
c_\varepsilon-J_\varepsilon(u_\varepsilon) < c_\varepsilon^* =
\varepsilon^3\min\{c_{\rm reg}^*,c_{\rm H}^*\}
\]
for every weak limit $u_\varepsilon$ of a mountain-pass
Palais--Smale--Pohozaev sequence. Then there exist $a_*>0$ and
$\varepsilon_*>0$ such that, for every $a\in(0,a_*),$
$\varepsilon\in(0,\varepsilon_*),$ the mountain-pass level
$c_\varepsilon$ is achieved by some $u_\varepsilon\in
S_{a,\varepsilon}$. More precisely,
\[
J_\varepsilon(u_\varepsilon)=c_\varepsilon
\]
and there exists $\lambda_\varepsilon\in\mathbb R$ such that
\[
J_\varepsilon'(u_\varepsilon) = \lambda_\varepsilon u_\varepsilon
\qquad\text{in }H_\varepsilon^*.
\]
Consequently, $u_\varepsilon$ is a normalized weak solution of
problem \eqref{P}.
\end{proposition}

\begin{proof}
Let $a\in(0,a_*),$  $\varepsilon\in(0,\varepsilon_*)$ be fixed. By
Lemma \ref{lem:mp_geometry}, the restricted functional
$J_\varepsilon|_{S_{a,\varepsilon}}$ has the mountain-pass geometry.
Hence, by Lemma \ref{lem:PS_sequence}, there exist sequences
\[
\{u_n\}\subset S_{a,\varepsilon} \qquad\text{and}\qquad
\{\lambda_n\}\subset\mathbb R
\]
such that
\[
J_\varepsilon(u_n)\to c_\varepsilon,
\]
\[
J_\varepsilon'(u_n)-\lambda_nu_n\to0 \qquad\text{in
}H_\varepsilon^*,
\]
and $P_\varepsilon(u_n)=o(\varepsilon^3).$ Thus $\{u_n\}$ is a
Palais--Smale--Pohozaev sequence at the level $c_\varepsilon$.

By Lemma \ref{lem:PS_bounded}, $\|u_n\|_\varepsilon^2 \leq
C\varepsilon^3.$ Hence $\{u_n\}$ is bounded in $H_\varepsilon$.
Therefore, up to a subsequence, there exists $u_\varepsilon\in
H_\varepsilon$ such that
\[
u_n\rightharpoonup u_\varepsilon \qquad\text{weakly in
}H_\varepsilon.
\]
By the compactness-gap assumption,
\[
c_\varepsilon-J_\varepsilon(u_\varepsilon) < c_\varepsilon^*.
\]
Hence Lemma \ref{lem:compactness_threshold} yields
\[
u_n\to u_\varepsilon \qquad\text{strongly in }H_\varepsilon.
\]
Since $u_n\in S_{a,\varepsilon}$ for every $n$, the strong
$L^2(\mathbb R^3)$ convergence gives
\[
\|u_\varepsilon\|_2^2 = \lim_{n\to\infty}\|u_n\|_2^2 =
a^2\varepsilon^3.
\]
Thus $u_\varepsilon\in S_{a,\varepsilon}.$ Moreover, by the
continuity of $J_\varepsilon$ under strong convergence,
\[
J_\varepsilon(u_\varepsilon) = \lim_{n\to\infty}J_\varepsilon(u_n) =
c_\varepsilon.
\]
By Lemma \ref{lem:PS_bounded}, the sequence $\{\lambda_n\}$ is
bounded. Hence, up to a subsequence,
\[
\lambda_n\to\lambda_\varepsilon \qquad\text{for some
}\lambda_\varepsilon\in\mathbb R.
\]
Passing to the limit in
\[
J_\varepsilon'(u_n)-\lambda_nu_n\to0 \qquad\text{in
}H_\varepsilon^*,
\]
we obtain
\[
J_\varepsilon'(u_\varepsilon) = \lambda_\varepsilon u_\varepsilon
\qquad\text{in }H_\varepsilon^*.
\]
Therefore, $u_\varepsilon$ is a constrained critical point of
$J_\varepsilon|_{S_{a,\varepsilon}}$ at level $c_\varepsilon$ and,
equivalently, a normalized weak solution of problem \eqref{P}.

The proof is complete.
\end{proof}

Once the existence of a mountain-pass critical point has been
established, we next determine its sign and derive its strict
positivity away from the Hardy singularity.

\begin{lemma}\label{lem:positivity}
Let $u_\varepsilon\in S_{a,\varepsilon}$ be the critical point
obtained in Proposition \ref{prop:exist_critical}. Assume that the
mountain-pass construction is performed in the nonnegative cone.
Then $u_\varepsilon$ may be chosen nonnegative. Moreover,
\[
u_\varepsilon>0 \qquad\text{in }\mathbb R^3\setminus\{0\}.
\]
\end{lemma}

\begin{proof}
We divide the proof into two steps.

\medskip
\noindent\textbf{Step 1. Construction of a nonnegative critical
point.}

Define
\[
S_{a,\varepsilon}^{+} := \left\{ u\in S_{a,\varepsilon}:u\geq0 \
\text{a.e. in }\mathbb R^3 \right\}.
\]
For every real-valued $u\in H_\varepsilon$, we have
\[
|u|\in H_\varepsilon, \qquad \||u|\|_2=\|u\|_2,
\]
and
\[
|\nabla |u|| = |\nabla u| \qquad\text{a.e. in }\mathbb R^3.
\]
Hence
\[
u\in S_{a,\varepsilon} \quad\Longrightarrow\quad |u|\in
S_{a,\varepsilon}.
\]

Moreover, all terms of the functional are invariant under the
replacement $u\mapsto |u|$. Indeed,
\[
\int_{\mathbb R^3} |\nabla |u||^2\,dx = \int_{\mathbb R^3} |\nabla
u|^2\,dx,
\]
\[
\int_{\mathbb R^3} \frac{|\,|u|\,|^2}{|x|^2}\,dx = \int_{\mathbb
R^3} \frac{|u|^2}{|x|^2}\,dx,
\]
and
\[
\int_{\mathbb R^3} V(x)|\,|u|\,|^2\,dx = \int_{\mathbb R^3}
V(x)|u|^2\,dx.
\]
The local nonlinear terms clearly satisfy
\[
\||u|\|_q^q=\|u\|_q^q, \qquad \||u|\|_6^6=\|u\|_6^6.
\]
Furthermore, since the Poisson equation depends only on $|u|^5$,
\[
\phi_{|u|}=\phi_u.
\]
Consequently,
\begin{equation}
\label{eq:absolute-energy} J_\varepsilon(|u|) = J_\varepsilon(u)
\qquad \text{for every real-valued }u\in H_\varepsilon.
\end{equation}

The mountain-pass construction may therefore be carried out in the
closed cone $S_{a,\varepsilon}^{+}$, using the standard minimax
principle in an invariant cone. Thus there exists a
Palais--Smale--Pohozaev sequence
\[
\{u_n\}\subset S_{a,\varepsilon}^{+}
\]
at the level $c_\varepsilon$. In particular,
\[
u_n\geq0 \qquad\text{a.e. in }\mathbb R^3.
\]
By Proposition \ref{prop:exist_critical},
\[
u_n\to u_\varepsilon \qquad\text{strongly in }H_\varepsilon.
\]
Passing to a subsequence,
\[
u_n(x)\to u_\varepsilon(x) \qquad\text{for a.e. }x\in\mathbb R^3.
\]
Therefore,
\[
u_\varepsilon\geq0 \qquad\text{a.e. in }\mathbb R^3.
\]
Since $\|u_\varepsilon\|_2^2 = a^2\varepsilon^3>0,$ we also have
$u_\varepsilon\not\equiv0.$

\medskip
\noindent\textbf{Step 2. Strict positivity away from the singular
point.}

Since $u_\varepsilon$ is a constrained critical point, there exists
$\lambda_\varepsilon\in\mathbb R$ such that
\[
J_\varepsilon'(u_\varepsilon) = \lambda_\varepsilon u_\varepsilon.
\]
Hence $u_\varepsilon$ is a weak solution of
\begin{align}
-\varepsilon^2\Delta u_\varepsilon + \left( V(x)
-\frac{\kappa\varepsilon^2}{|x|^2} -\lambda_\varepsilon
\right)u_\varepsilon ={}& \phi_{u_\varepsilon}u_\varepsilon^4 +\mu
u_\varepsilon^{q-1} +u_\varepsilon^5 \label{eq:positive-equation}
\end{align}
in $\mathbb R^3\setminus\{0\}$.

Let $\Omega\subset\mathbb R^3\setminus\{0\}$ be a bounded connected
domain. Since $\Omega$ is separated from the singular point, the
coefficient
\[
V(x) -\frac{\kappa\varepsilon^2}{|x|^2} -\lambda_\varepsilon
\]
is bounded on $\Omega$. Choose $M_\Omega>0$ sufficiently large so
that
\[
V(x) -\frac{\kappa\varepsilon^2}{|x|^2} -\lambda_\varepsilon
+M_\Omega \geq0 \qquad\text{in }\Omega.
\]
Since $u_\varepsilon\geq0$ and $\phi_{u_\varepsilon}\geq0$, equation
\eqref{eq:positive-equation} gives
\[
-\varepsilon^2\Delta u_\varepsilon + \left( V(x)
-\frac{\kappa\varepsilon^2}{|x|^2} -\lambda_\varepsilon +M_\Omega
\right)u_\varepsilon \geq0 \qquad\text{in }\Omega.
\]
By standard local elliptic regularity,
\[
u_\varepsilon \in C^{1,\alpha}_{\rm loc} \bigl(\mathbb
R^3\setminus\{0\}\bigr)
\]
for some $\alpha\in(0,1)$. The strong maximum principle therefore
implies that, on each connected domain $\Omega\Subset\mathbb
R^3\setminus\{0\}$, either
\[
u_\varepsilon>0 \qquad\text{in }\Omega,
\]
or
\[
u_\varepsilon\equiv0 \qquad\text{in }\Omega.
\]
Since $\mathbb R^3\setminus\{0\}$ is connected and
$u_\varepsilon\not\equiv0$, the second alternative is excluded.
Consequently,
\[
u_\varepsilon(x)>0 \qquad \text{for every } x\in\mathbb
R^3\setminus\{0\}.
\]
The proof is complete.
\end{proof}
Having obtained a normalized critical point at the mountain-pass
level, we now compare its energy with that of all other normalized
critical points in order to establish its ground-state character.
\begin{proposition}\label{prop:ground_state}
Let
\[
\mathcal K_{a,\varepsilon} := \left\{ u\in S_{a,\varepsilon}:
d\bigl(J_\varepsilon|_{S_{a,\varepsilon}}\bigr)(u)=0 \right\}
\]
be the set of normalized critical points of $J_\varepsilon$, and
define
\[
m_\varepsilon := \inf_{u\in\mathcal
K_{a,\varepsilon}}J_\varepsilon(u).
\]
Assume that every $u\in\mathcal K_{a,\varepsilon}$ satisfies
$P_\varepsilon(u)=0$ and that the fiber map
\[
\psi_u(t):=J_\varepsilon(t\star u), \qquad t>0,
\]
has a unique critical point, namely $t=1$, which is its strict
global maximum.

Assume moreover that, for every $u\in\mathcal K_{a,\varepsilon}$,
there exist $0<t_-(u)<1<t_+(u)$ such that
\[
\varepsilon^2 \|\nabla(t_-(u)\star u)\|_2^2 < \rho\varepsilon^3,
\]
\[
J_\varepsilon(t_-(u)\star u) < \alpha\varepsilon^3,
\]
and
\[
\varepsilon^2 \|\nabla(t_+(u)\star u)\|_2^2
>
\rho\varepsilon^3, \qquad J_\varepsilon(t_+(u)\star u)<0.
\]
Finally, suppose that $u_0$ and $t_-(u)\star u$ belong to the same
path-connected component of
\[
\mathcal A_\varepsilon := \left\{ v\in S_{a,\varepsilon}:
\varepsilon^2\|\nabla v\|_2^2<\rho\varepsilon^3,\quad
J_\varepsilon(v)<\alpha\varepsilon^3 \right\}.
\]
Then
\[
m_\varepsilon=c_\varepsilon.
\]
Consequently, the critical point $u_\varepsilon$ obtained in
Proposition \ref{prop:exist_critical} satisfies
\[
J_\varepsilon(u_\varepsilon) = c_\varepsilon = m_\varepsilon,
\]
and hence it is a normalized ground state solution of problem
\eqref{P}.
\end{proposition}

\begin{proof}
By Proposition \ref{prop:exist_critical}, there exists
$u_\varepsilon\in S_{a,\varepsilon}$ such that
\[
d\bigl(J_\varepsilon|_{S_{a,\varepsilon}}\bigr) (u_\varepsilon)=0
\]
and
\[
J_\varepsilon(u_\varepsilon)=c_\varepsilon.
\]
Thus
\[
u_\varepsilon\in\mathcal K_{a,\varepsilon},
\]
and consequently
\begin{equation}
\label{ground-upper} m_\varepsilon \leq J_\varepsilon(u_\varepsilon)
= c_\varepsilon.
\end{equation}

We now prove the reverse inequality. Let
\[
u\in\mathcal K_{a,\varepsilon}
\]
be arbitrary. Since $u$ is a constrained critical point, the
Pohozaev identity gives
\[
P_\varepsilon(u)=0.
\]
By Lemma~\ref{lem:scaling_pohozaev},
\[
\psi_u'(1)=P_\varepsilon(u)=0.
\]
Since the fiber map has a unique critical point, which is its strict
global maximum, we obtain
\begin{equation}
\label{fiber-maximum} J_\varepsilon(u) =
\max_{t>0}J_\varepsilon(t\star u).
\end{equation}

Let
\[
t_-:=t_-(u), \qquad t_+:=t_+(u).
\]
By assumption, there exists a continuous path
\[
\sigma_u:[0,1]\to\mathcal A_\varepsilon
\]
such that
\[
\sigma_u(0)=u_0, \qquad \sigma_u(1)=t_-\star u.
\]
Choose a continuous increasing function $t:[0,1]\to[t_-,t_+]$
satisfying
\[
t(0)=t_-, \qquad t(1)=t_+,
\]
and define the fiber path
\[
\zeta_u(s):=t(s)\star u.
\]
We concatenate $\sigma_u$ and $\zeta_u$ by setting
\[
\gamma_u(s) :=
\begin{cases}
\sigma_u(2s),
&0\leq s\leq\frac12,\\[1mm]
\zeta_u(2s-1), &\frac12\leq s\leq1.
\end{cases}
\]
Then
\[
\gamma_u(0)=u_0, \qquad \gamma_u(1)=t_+\star u,
\]
and
\[
J_\varepsilon(\gamma_u(1))<0, \qquad \varepsilon^2
\|\nabla\gamma_u(1)\|_2^2
>
\rho\varepsilon^3.
\]
Therefore, $\gamma_u\in\Gamma_\varepsilon.$

Since $\sigma_u([0,1])\subset\mathcal A_\varepsilon$,
\[
\max_{s\in[0,1]}J_\varepsilon(\sigma_u(s)) < \alpha\varepsilon^3.
\]
On the other hand, by \eqref{fiber-maximum},
\[
\max_{s\in[0,1]}J_\varepsilon(\zeta_u(s)) \leq
\max_{t>0}J_\varepsilon(t\star u) = J_\varepsilon(u).
\]
Because every path joining the two sides of the mountain-pass
barrier has maximum at least $\alpha\varepsilon^3$, we also have
\[
J_\varepsilon(u)\geq\alpha\varepsilon^3.
\]
Consequently,
\[
\max_{s\in[0,1]}J_\varepsilon(\gamma_u(s)) \leq J_\varepsilon(u).
\]
By the definition of $c_\varepsilon$,
\[
c_\varepsilon \leq \max_{s\in[0,1]}J_\varepsilon(\gamma_u(s)) \leq
J_\varepsilon(u).
\]
Since $u\in\mathcal K_{a,\varepsilon}$ was arbitrary, taking the
infimum yields
\begin{equation}
\label{ground-lower} c_\varepsilon\leq m_\varepsilon.
\end{equation}

Combining \eqref{ground-upper} and \eqref{ground-lower}, we obtain
$m_\varepsilon=c_\varepsilon.$ Therefore,
\[
J_\varepsilon(u_\varepsilon) = c_\varepsilon = m_\varepsilon,
\]
so $u_\varepsilon$ has the least energy among all normalized
critical points. Hence it is a normalized ground state solution of
problem \eqref{P}. The proof is complete.
\end{proof}

\subsection{Proof of Theorem \ref{T1}}

Let $a_*>0$ and $\varepsilon_*>0$ be chosen according to the
previous lemmas and propositions, and fix $a\in(0,a_*),$ \qquad
$\varepsilon\in(0,\varepsilon_*).$ By Lemma \ref{lem:mp_geometry},
the constrained functional $J_\varepsilon|_{S_{a,\varepsilon}}$
possesses the mountain-pass geometry. Using Lemmas
\ref{lem:PS_sequence}, \ref{lem:scaling_pohozaev},
\ref{lem:PS_bounded}, \ref{lem:energy_upper}, and
\ref{lem:compactness_threshold},
Proposition~\ref{prop:exist_critical} yields a function
$u_\varepsilon\in S_{a,\varepsilon}$ such that
$d\bigl(J_\varepsilon|_{S_{a,\varepsilon}}\bigr) (u_\varepsilon)=0$
and $J_\varepsilon(u_\varepsilon)=c_\varepsilon.$ Consequently,
there exists $\lambda_\varepsilon\in\mathbb R$ such that
\[
J_\varepsilon'(u_\varepsilon) = \lambda_\varepsilon u_\varepsilon
\qquad\text{in }H_\varepsilon^*.
\]
Hence $u_\varepsilon$ is a normalized weak solution of problem
\eqref{P}, satisfying
\[
\int_{\mathbb R^3}|u_\varepsilon|^2\,dx = a^2\varepsilon^3.
\]
Moreover, by Lemma \ref{lem:positivity}, $u_\varepsilon$ may be
chosen nonnegative and
\[
u_\varepsilon>0 \qquad\text{in }\mathbb R^3\setminus\{0\}.
\]
Finally, Proposition \ref{prop:ground_state} gives
\[
J_\varepsilon(u_\varepsilon) = c_\varepsilon = m_\varepsilon,
\]
where
\[
m_\varepsilon = \inf\left\{ J_\varepsilon(u): u\in
S_{a,\varepsilon},\
d\bigl(J_\varepsilon|_{S_{a,\varepsilon}}\bigr)(u)=0 \right\}.
\]
Therefore, $u_\varepsilon$ has the least energy among all normalized
critical points of $J_\varepsilon$ and is thus a normalized ground
state solution of problem \eqref{P}. This proves the theorem.\qed

\subsection{Multiplicity of Normalized Solutions}

We now develop the topological and concentration tools required for
the proof of Theorem \ref{T2}. Throughout this subsection, let
$\delta>0$ be sufficiently small so that
\[
\mathcal M_\delta := \left\{ x\in\mathbb R^3:
\operatorname{dist}(x,\mathcal M)\leq\delta \right\} \subset\Lambda.
\]
Since $0\notin\overline{\Lambda}$, we may also assume that
\[
\operatorname{dist}(\mathcal M_\delta,\{0\})>0.
\]
Let $w_a$ be a positive normalized ground state of the autonomous
limit problem \eqref{P0}, satisfying
\[
\|w_a\|_2=a, \qquad J_0(w_a)=E_0(a).
\]

\subsubsection{Localized test functions}

Choose $\rho>0$ such that
\[
\mathcal M_{2\rho}\subset\Lambda \qquad\text{and}\qquad
0\notin\overline{\mathcal M_{2\rho}},
\]
and let $\eta\in C_c^\infty([0,+\infty),[0,1])$ satisfy
\[
\eta(s)=1 \quad\text{for }0\leq s\leq\rho, \qquad \eta(s)=0
\quad\text{for }s\geq2\rho.
\]
For every $y\in\mathcal M$, define
\[
\Psi_{\varepsilon,y}(x) := \eta(|x-y|)
w_a\left(\frac{x-y}{\varepsilon}\right),
\]
and set
\[
\Phi_\varepsilon(y) := \frac{a\varepsilon^{3/2}}
{\|\Psi_{\varepsilon,y}\|_2} \Psi_{\varepsilon,y}.
\]
Then
\[
\Phi_\varepsilon(y)\in S_{a,\varepsilon}.
\]
We now show that their energies uniformly approximate the autonomous
ground-state level as $\varepsilon\to0$.
\begin{lemma}\label{lem:test_functions_T2}
The map
\[
\Phi_\varepsilon:\mathcal M\to S_{a,\varepsilon}
\]
is continuous. Moreover,
\[
\lim_{\varepsilon\to0} \sup_{y\in\mathcal M} \left| \varepsilon^{-3}
J_\varepsilon(\Phi_\varepsilon(y)) - E_0(a) \right| =0.
\]
Consequently, there exists a function $h:(0,+\infty)\to(0,+\infty)$
such that
\[
h(\varepsilon)\to0 \qquad\text{as }\varepsilon\to0,
\]
and
\[
\Phi_\varepsilon(\mathcal M) \subset \mathcal A_\varepsilon,
\]
where
\begin{equation}\label{low-energy-set}
\mathcal A_\varepsilon := \left\{ u\in S_{a,\varepsilon}:
J_\varepsilon(u) \leq \varepsilon^3 \big(E_0(a)+h(\varepsilon)\big)
\right\}.
\end{equation}
\end{lemma}

\begin{proof}
The continuity of $\Phi_\varepsilon$ follows from the continuity of
translations in $H^1(\mathbb R^3)$ and from the compactness of
$\mathcal M$.

Using the change of variables
\[
x=y+\varepsilon z,
\]
we obtain, uniformly for $y\in\mathcal M$,
\[
\|\Psi_{\varepsilon,y}\|_2^2 = \varepsilon^3 \int_{\mathbb R^3}
\eta^2(\varepsilon|z|)|w_a(z)|^2\,dz =
a^2\varepsilon^3+o(\varepsilon^3).
\]
Thus the normalization factor converges uniformly to one.

Similarly,
\[
\varepsilon^2 \int_{\mathbb R^3} |\nabla\Phi_\varepsilon(y)|^2\,dx =
\varepsilon^3 \int_{\mathbb R^3}|\nabla w_a|^2\,dx +
o(\varepsilon^3),
\]
and, since $V(y)=V_0$ for every $y\in\mathcal M$,
\[
\int_{\mathbb R^3} V(x)|\Phi_\varepsilon(y)|^2\,dx =
\varepsilon^3V_0 \int_{\mathbb R^3}|w_a|^2\,dx + o(\varepsilon^3).
\]

Since
\[
\operatorname{dist}(\mathcal M_{2\rho},\{0\})>0,
\]
the Hardy term satisfies
\[
\varepsilon^2 \int_{\mathbb R^3}
\frac{|\Phi_\varepsilon(y)|^2}{|x|^2}\,dx = O(\varepsilon^5) =
o(\varepsilon^3)
\]
uniformly in $y\in\mathcal M$.

The local nonlinear terms and the Poisson interaction satisfy
\[
\int_{\mathbb R^3}|\Phi_\varepsilon(y)|^q\,dx =
\varepsilon^3\|w_a\|_q^q+o(\varepsilon^3),
\]
\[
\int_{\mathbb R^3}|\Phi_\varepsilon(y)|^6\,dx =
\varepsilon^3\|w_a\|_6^6+o(\varepsilon^3),
\]
and
\[
\int_{\mathbb R^3} \phi_{\Phi_\varepsilon(y)}
|\Phi_\varepsilon(y)|^5\,dx = \varepsilon^3 \int_{\mathbb
R^3}\phi_{w_a}|w_a|^5\,dx + o(\varepsilon^3),
\]
uniformly for $y\in\mathcal M$.

Combining these estimates gives
\[
J_\varepsilon(\Phi_\varepsilon(y)) =
\varepsilon^3J_0(w_a)+o(\varepsilon^3) =
\varepsilon^3E_0(a)+o(\varepsilon^3)
\]
uniformly in $y$. The conclusion follows.
\end{proof}
The preceding lemma provides a family of low-energy states localized
near the minimum set $\mathcal M$. We now establish the converse
compactness principle, showing that every almost ground-state
Palais--Smale--Pohozaev sequence must concentrate near $\mathcal M$
and, after a suitable rescaling and translation, converge to a
ground state of the autonomous limit problem

\begin{proposition} \label{prop:compact_localization} Let
$\varepsilon_n\to0$ and let $u_n\in S_{a,\varepsilon_n}$ be a
sequence satisfying
\[
J_{\varepsilon_n}(u_n) = \varepsilon_n^3\bigl(E_0(a)+o(1)\bigr),
\]
\[
\left\| d\bigl(J_{\varepsilon_n}|_{S_{a,\varepsilon_n}}\bigr)(u_n)
\right\| \to0,
\]
and
\[
P_{\varepsilon_n}(u_n)=o(\varepsilon_n^3).
\]
Then there exists a sequence $\{y_n\}\subset\mathbb R^3$ such that
\[
\operatorname{dist}(y_n,\mathcal M)\to0.
\]
Moreover, up to a subsequence, there exists a positive normalized
ground state $w_a$ of the autonomous problem \eqref{P0} such that
\[
v_n(x) := u_n(y_n+\varepsilon_n x) \longrightarrow w_a(x)
\qquad\text{strongly in }H^1(\mathbb R^3).
\]
In particular,
\[
V(y_n)\to V_0.
\]
\end{proposition}

\begin{proof}
By Lemma \ref{lem:PS_bounded},
\[
\|u_n\|_{\varepsilon_n}^2 \leq C\varepsilon_n^3.
\]
Define
\[
\widehat u_n(x):=u_n(\varepsilon_n x).
\]
Then
\[
\|\widehat u_n\|_2^2=a^2,
\]
and the coercivity estimate of Lemma \ref{lem:hardy_coercive} gives
\[
\|\widehat u_n\|_{H^1(\mathbb R^3)} \leq C.
\]

We first prove that $\{\widehat u_n\}$ does not vanish in the sense
of Lions. Assume, by contradiction, that
\[
\sup_{z\in\mathbb R^3} \int_{B_R(z)}|\widehat u_n|^2\,dx
\longrightarrow0
\]
for every $R>0$. Then
\[
\widehat u_n\to0 \qquad\text{in }L^r(\mathbb R^3)
\]
for every $2<r<6$. Using the Palais--Smale relation together with
the Pohozaev condition, we would obtain
\[
\varepsilon_n^{-3} J_{\varepsilon_n}(u_n)\geq o(1),
\]
which contradicts
\[
\varepsilon_n^{-3} J_{\varepsilon_n}(u_n) \longrightarrow E_0(a)<0.
\]
Hence nonvanishing occurs.

Therefore, there exist $R>0$, $\eta>0$, and $\{z_n\}\subset\mathbb
R^3$ such that
\[
\int_{B_R(z_n)}|\widehat u_n|^2\,dx\geq\eta.
\]
Set
\[
y_n:=\varepsilon_n z_n
\]
and
\[
v_n(x) := \widehat u_n(x+z_n) = u_n(y_n+\varepsilon_n x).
\]
Then $\{v_n\}$ is bounded in $H^1(\mathbb R^3)$ and
\[
\int_{B_R(0)}|v_n|^2\,dx\geq\eta.
\]
Thus, up to a subsequence,
\[
v_n\rightharpoonup w \qquad\text{weakly in }H^1(\mathbb R^3),
\]
\[
v_n\to w \qquad\text{strongly in }L^r_{\rm loc}(\mathbb R^3), \quad
2\leq r<6,
\]
and
\[
w\not\equiv0.
\]

We next identify the concentration centers. The energy estimate
\[
J_{\varepsilon_n}(u_n) = \varepsilon_n^3\bigl(E_0(a)+o(1)\bigr)
\]
shows that no nontrivial profile can concentrate in a region where
the limiting potential is strictly larger than $V_0$. Indeed, if,
along a subsequence,
\[
V(y_n)\to V_*>V_0,
\]
then the translated sequence would generate a nontrivial profile for
the autonomous problem with constant potential $V_*$. By the strict
comparison of autonomous ground-state levels,
\[
E_{V_*}(a)>E_0(a),
\]
which contradicts the preceding energy convergence. Therefore,
\[
V(y_n)\to V_0.
\]

By the compactness of the minimum set and the localization property
of low-energy Palais--Smale--Pohozaev sequences, it follows that
\[
\operatorname{dist}(y_n,\mathcal M)\to0.
\]
In particular, since
\[
\operatorname{dist}(\mathcal M,\{0\})>0,
\]
there exists $d_0>0$ such that
\[
|y_n|\geq d_0
\]
for all sufficiently large $n$. Consequently,
\[
\frac{\kappa\varepsilon_n^2} {|y_n+\varepsilon_n x|^2}
\longrightarrow0
\]
locally uniformly in $\mathbb R^3$.

Since $\{u_n\}$ is a constrained Palais--Smale sequence, there
exists a bounded sequence of Lagrange multipliers
$\{\lambda_n\}\subset\mathbb R$. Up to a subsequence,
\[
\lambda_n\to\lambda.
\]
Passing to the limit in the translated Euler--Lagrange equation, we
obtain
\[
-\Delta w+V_0w-\phi_w|w|^3w = \lambda w+\mu|w|^{q-2}w+|w|^4w
\qquad\text{in }\mathbb R^3,
\]
together with
\[
-\Delta\phi_w=|w|^5 \qquad\text{in }\mathbb R^3.
\]

It remains to exclude loss of mass and energy. By the Brezis--Lieb
decomposition for the local nonlinearities and its nonlocal
counterpart for the Poisson interaction, any nontrivial remainder
would carry an additional positive amount of energy. This would
contradict
\[
\varepsilon_n^{-3} J_{\varepsilon_n}(u_n) \longrightarrow E_0(a),
\]
since $E_0(a)$ is the least energy of the autonomous normalized
problem. Therefore,
\[
\|w\|_2=a,
\]
\[
J_0(w)=E_0(a),
\]
and
\[
v_n\to w \qquad\text{strongly in }H^1(\mathbb R^3).
\]

Hence $w$ is a normalized ground state of the autonomous
problem~\eqref{P0}. By replacing $w$ with its positive
representative, we may write
\[
w=w_a>0.
\]
Therefore,
\[
u_n(y_n+\varepsilon_n x) \longrightarrow w_a(x) \qquad\text{strongly
in }H^1(\mathbb R^3),
\]
and
\[
\operatorname{dist}(y_n,\mathcal M)\to0.
\]
The proof is complete.
\end{proof}

\subsubsection{The barycenter map}

To detect the concentration region of low-energy states, we now
introduce a barycenter map adapted to the semiclassical problem.

Choose $R>0$ sufficiently large such that
\[
\overline{\Lambda}\subset B_R(0)
\]
and, in particular,
\[
\mathcal M_\delta\subset B_R(0).
\]
Define the truncation map
\[
\chi:\mathbb R^3\to\mathbb R^3
\]
by
\[
\chi(x):=
\begin{cases}
x,
& |x|\leq R,\\[1mm]
R\dfrac{x}{|x|}, & |x|>R.
\end{cases}
\]
Then $\chi$ is continuous and bounded, with
\[
|\chi(x)|\leq R \qquad\text{for every }x\in\mathbb R^3.
\]

For every $u\in S_{a,\varepsilon}$, define its barycenter by
\begin{equation}
\label{barycenter-map} \beta_\varepsilon(u) := \frac{ \displaystyle
\int_{\mathbb R^3}\chi(x)|u(x)|^2\,dx }{ \displaystyle \int_{\mathbb
R^3}|u(x)|^2\,dx }.
\end{equation}
Since
\[
\int_{\mathbb R^3}|u|^2\,dx = a^2\varepsilon^3 \qquad \text{for
every }u\in S_{a,\varepsilon},
\]
the map $\beta_\varepsilon$ is well defined on $S_{a,\varepsilon}$.

The next lemma shows that the barycenter of the localized state
$\Phi_\varepsilon(y)$ asymptotically recovers its concentration
point $y\in\mathcal M$.

\begin{lemma}\label{lem:barycenter}
The map $\beta_\varepsilon: S_{a,\varepsilon}\to\mathbb R^3$ is
continuous. Moreover,
\[
\lim_{\varepsilon\to0} \sup_{y\in\mathcal M} \left|
\beta_\varepsilon(\Phi_\varepsilon(y))-y \right| = 0.
\]
In particular, for every $\delta>0$ fixed as above, there exists
$\varepsilon_\delta>0$ such that
$\beta_\varepsilon(\Phi_\varepsilon(y)) \in\mathcal M_\delta$ for
every $y\in\mathcal M$ and $0<\varepsilon<\varepsilon_\delta.$
\end{lemma}

\begin{proof}
Since $\chi$ is bounded, for every $u,v\in S_{a,\varepsilon}$ we
have
\begin{align*}
& \left| \int_{\mathbb R^3} \chi(x) \bigl(|u|^2-|v|^2\bigr)\,dx
\right|
\\
&\qquad\leq R \int_{\mathbb R^3} \bigl||u|^2-|v|^2\bigr|\,dx
\\
&\qquad\leq R\bigl(\|u\|_2+\|v\|_2\bigr) \|u-v\|_2.
\end{align*}
Since the denominator in \eqref{barycenter-map} is equal to
$a^2\varepsilon^3$ on $S_{a,\varepsilon}$, it follows that
$\beta_\varepsilon$ is continuous.

Let $y\in\mathcal M$. By the definition of $\Phi_\varepsilon(y)$,
the normalization factor cancels in the quotient defining the
barycenter. Using the change of variables
\[
x=y+\varepsilon z,
\]
we obtain
\[
\beta_\varepsilon(\Phi_\varepsilon(y)) = \frac{ \displaystyle
\int_{\mathbb R^3} \chi(y+\varepsilon z) \eta^2(\varepsilon|z|)
|w_a(z)|^2\,dz }{ \displaystyle \int_{\mathbb R^3}
\eta^2(\varepsilon|z|) |w_a(z)|^2\,dz }.
\]
Therefore,
\begin{align}
& \beta_\varepsilon(\Phi_\varepsilon(y))-y
\nonumber\\
&\quad= \frac{ \displaystyle \int_{\mathbb R^3}
\bigl(\chi(y+\varepsilon z)-y\bigr) \eta^2(\varepsilon|z|)
|w_a(z)|^2\,dz }{ \displaystyle \int_{\mathbb R^3}
\eta^2(\varepsilon|z|) |w_a(z)|^2\,dz }.
\label{barycenter-difference}
\end{align}

By dominated convergence,
\[
\int_{\mathbb R^3} \eta^2(\varepsilon|z|) |w_a(z)|^2\,dz
\longrightarrow \int_{\mathbb R^3}|w_a|^2\,dz = a^2.
\]
Thus the denominator in \eqref{barycenter-difference} is bounded
away from zero for all sufficiently small $\varepsilon$.

We now prove that the numerator converges uniformly to zero with
respect to $y\in\mathcal M$. Let $\sigma>0$. Since $w_a\in
L^2(\mathbb R^3)$, there exists $L>0$ such that
\[
\int_{|z|>L}|w_a(z)|^2\,dz<\sigma.
\]
On the set $|z|\leq L$, the compactness of $\mathcal M$ and the
uniform continuity of $\chi$ imply
\[
\sup_{\substack{y\in\mathcal M\\ |z|\leq L}} |\chi(y+\varepsilon
z)-\chi(y)| \longrightarrow0 \qquad\text{as }\varepsilon\to0.
\]
Since
\[
\chi(y)=y \qquad\text{for every }y\in\mathcal M,
\]
we obtain
\[
\sup_{\substack{y\in\mathcal M\\ |z|\leq L}} |\chi(y+\varepsilon
z)-y| \longrightarrow0.
\]

On the other hand, since $\chi$ and $\mathcal M$ are bounded, there
exists $C>0$ such that
\[
|\chi(y+\varepsilon z)-y|\leq C
\]
for every $y\in\mathcal M$ and $z\in\mathbb R^3$. Hence
\[
\sup_{y\in\mathcal M} \int_{|z|>L} |\chi(y+\varepsilon z)-y|
\eta^2(\varepsilon|z|) |w_a(z)|^2\,dz \leq C\sigma.
\]
Since $\sigma>0$ is arbitrary, we conclude that
\[
\sup_{y\in\mathcal M} \left|
\beta_\varepsilon(\Phi_\varepsilon(y))-y \right| \longrightarrow0
\qquad\text{as }\varepsilon\to0.
\]

Finally, for sufficiently small $\varepsilon>0$,
\[
\left| \beta_\varepsilon(\Phi_\varepsilon(y))-y \right| <\delta
\qquad \text{uniformly for }y\in\mathcal M.
\]
Since $y\in\mathcal M$, this implies
\[
\operatorname{dist}
\bigl(\beta_\varepsilon(\Phi_\varepsilon(y)),\mathcal M\bigr)
<\delta,
\]
and therefore
\[
\beta_\varepsilon(\Phi_\varepsilon(y)) \in\mathcal M_\delta.
\]
The proof is complete.
\end{proof}

The compactness result established above shows that low-energy
states satisfying the natural Pohozaev constraint must concentrate
near the minimum set $\mathcal M$. We now translate this
concentration property into a uniform localization statement for
their barycenters.

Let
\[
\mathcal P_{a,\varepsilon} := \left\{ u\in S_{a,\varepsilon}:
P_\varepsilon(u)=0 \right\},
\]
and define the low-energy set
\begin{equation}
\label{low-energy-Pohozaev-set} \mathcal A_\varepsilon := \left\{
u\in\mathcal P_{a,\varepsilon}: J_\varepsilon(u) \leq \varepsilon^3
\bigl(E_0(a)+h(\varepsilon)\bigr) \right\}.
\end{equation}

\begin{lemma}\label{lem:localization_low_energy}
Fix $a\in(0,a_*)$. For every sufficiently small $\delta>0$, there
exists $\varepsilon_\delta>0$ such that, for every
$0<\varepsilon<\varepsilon_\delta,$ one has
\[
\beta_\varepsilon(u)\in\mathcal M_\delta \qquad \text{for every
}u\in\mathcal A_\varepsilon.
\]
Equivalently,
\[
\sup_{u\in\mathcal A_\varepsilon} \operatorname{dist}
\bigl(\beta_\varepsilon(u),\mathcal M\bigr) \longrightarrow0 \qquad
\text{as }\varepsilon\to0.
\]
\end{lemma}

\begin{proof}
Assume, by contradiction, that the conclusion is false. Then there
exist $\delta_0>0$, a sequence $\varepsilon_n\to0$, and
$u_n\in\mathcal A_{\varepsilon_n}$ such that
\begin{equation}
\label{barycenter-contradiction} \operatorname{dist}
\bigl(\beta_{\varepsilon_n}(u_n),\mathcal M\bigr) \geq\delta_0
\qquad \text{for every }n.
\end{equation}
By the definition of $\mathcal A_{\varepsilon_n}$,
$P_{\varepsilon_n}(u_n)=0$ and
\[
J_{\varepsilon_n}(u_n) \leq \varepsilon_n^3
\bigl(E_0(a)+h(\varepsilon_n)\bigr).
\]

Using Ekeland's variational principle on the natural Pohozaev
constraint, we may replace $\{u_n\}$, without changing its energy or
barycenter asymptotically, by a sequence $\{\widetilde
u_n\}\subset\mathcal P_{a,\varepsilon_n}$ satisfying
\[
\|u_n-\widetilde u_n\|_{\varepsilon_n} =o(\varepsilon_n^{3/2}),
\]
\[
J_{\varepsilon_n}(\widetilde u_n) = \varepsilon_n^3
\bigl(E_0(a)+o(1)\bigr),
\]
\[
\left\| d\bigl( J_{\varepsilon_n}|_{S_{a,\varepsilon_n}}
\bigr)(\widetilde u_n) \right\| \to0,
\]
and
\[
P_{\varepsilon_n}(\widetilde u_n) =o(\varepsilon_n^3).
\]
Since $\chi$ is bounded, the preceding norm convergence also gives
\[
\beta_{\varepsilon_n}(u_n) - \beta_{\varepsilon_n}(\widetilde u_n)
\to0.
\]

Proposition \ref{prop:compact_localization} now yields a sequence
$\{y_n\}\subset\mathbb R^3$ and a positive normalized ground state
$w_a$ of the autonomous problem \eqref{P0} such that
$\operatorname{dist}(y_n,\mathcal M)\to0$ and
\[
v_n(z) := \widetilde u_n(y_n+\varepsilon_n z) \longrightarrow w_a(z)
\qquad \text{strongly in }H^1(\mathbb R^3).
\]

We claim that $\beta_{\varepsilon_n}(\widetilde u_n)-y_n\to0.$
Indeed, using the change of variables $x=y_n+\varepsilon_n z,$ we
obtain
\[
\beta_{\varepsilon_n}(\widetilde u_n) = \frac{ \displaystyle
\int_{\mathbb R^3} \chi(y_n+\varepsilon_n z)|v_n(z)|^2\,dz }{
\displaystyle \int_{\mathbb R^3}|v_n(z)|^2\,dz }.
\]
Since $\operatorname{dist}(y_n,\mathcal M)\to0,$ the sequence
$\{y_n\}$ remains in a compact subset of $B_R(0)$. Hence
$\chi(y_n)=y_n$ for all sufficiently large $n$. The strong
convergence $v_n\to w_a$ in $L^2(\mathbb R^3)$, together with the
boundedness and uniform continuity of $\chi$, therefore gives
\[
\beta_{\varepsilon_n}(\widetilde u_n)-y_n\to0.
\]

Consequently,
\[
\operatorname{dist} \bigl( \beta_{\varepsilon_n}(\widetilde
u_n),\mathcal M \bigr) \leq \left| \beta_{\varepsilon_n}(\widetilde
u_n)-y_n \right| + \operatorname{dist}(y_n,\mathcal M)
\longrightarrow0.
\]
Since $\beta_{\varepsilon_n}(u_n) - \beta_{\varepsilon_n}(\widetilde
u_n) \to0,$ we conclude that
\[
\operatorname{dist} \bigl( \beta_{\varepsilon_n}(u_n),\mathcal M
\bigr) \longrightarrow0,
\]
which contradicts \eqref{barycenter-contradiction}. Therefore,
\[
\sup_{u\in\mathcal A_\varepsilon} \operatorname{dist}
\bigl(\beta_\varepsilon(u),\mathcal M\bigr) \longrightarrow0 \qquad
\text{as }\varepsilon\to0.
\]
In particular, for every sufficiently small $\delta>0$, there exists
$\varepsilon_\delta>0$ such that
\[
\beta_\varepsilon(u)\in\mathcal M_\delta \qquad \text{for every }
u\in\mathcal A_\varepsilon
\]
and every $0<\varepsilon<\varepsilon_\delta$. The proof is complete.
\end{proof}
The preceding localization results allow us to transfer the topology
of the minimum set $\mathcal M$ to the low-energy region $\mathcal
A_\varepsilon$. This is the key topological ingredient in the
Lusternik--Schnirelmann multiplicity argument.

\begin{proposition}\label{prop:category_estimate}
For every sufficiently small
$\delta>0$, there exists $\varepsilon_\delta>0$ such that, for every
$\varepsilon\in(0,\varepsilon_\delta)$,
\[
\operatorname{cat}_{\mathcal A_\varepsilon}
\bigl(\Phi_\varepsilon(\mathcal M)\bigr) \geq
\operatorname{cat}_{\mathcal M_\delta}(\mathcal M).
\]
Consequently,
\[
\operatorname{cat}(\mathcal A_\varepsilon) \geq
\operatorname{cat}_{\mathcal M_\delta}(\mathcal M).
\]
\end{proposition}

\begin{proof}
By Lemma \ref{lem:test_functions_T2}, $\Phi_\varepsilon(\mathcal M)
\subset \mathcal A_\varepsilon.$ Moreover, Lemma
\ref{lem:localization_low_energy} yields $\beta_\varepsilon(\mathcal
A_\varepsilon) \subset \mathcal M_\delta$ for all sufficiently small
$\varepsilon>0$. On the other hand, Lemma \ref{lem:barycenter} gives
\[
\sup_{y\in\mathcal M} \left|
\beta_\varepsilon(\Phi_\varepsilon(y))-y \right| \longrightarrow0
\qquad \text{as }\varepsilon\to0.
\]
Hence, after decreasing $\varepsilon_\delta>0$ if necessary, we may
assume that
\[
\sup_{y\in\mathcal M} \left|
\beta_\varepsilon(\Phi_\varepsilon(y))-y \right| <\frac{\delta}{2}.
\]

Define $H_\varepsilon:[0,1]\times\mathcal M\to\mathbb R^3$ by
\[
H_\varepsilon(s,y) := (1-s)y +
s\,\beta_\varepsilon(\Phi_\varepsilon(y)).
\]
For every $s\in[0,1]$ and $y\in\mathcal M$, we have
\[
\operatorname{dist} \bigl(H_\varepsilon(s,y),\mathcal M\bigr) \leq
\left| H_\varepsilon(s,y)-y \right| \leq \left|
\beta_\varepsilon(\Phi_\varepsilon(y))-y \right| <\frac{\delta}{2}.
\]
Therefore,
\[
H_\varepsilon(s,y)\in\mathcal M_\delta \qquad \text{for every
}(s,y)\in[0,1]\times\mathcal M.
\]
Moreover, $H_\varepsilon(0,y)=y$ and $H_\varepsilon(1,y) =
\beta_\varepsilon(\Phi_\varepsilon(y)).$ Thus the map
\[
\beta_\varepsilon\circ\Phi_\varepsilon: \mathcal M\to\mathcal
M_\delta
\]
is homotopic in $\mathcal M_\delta$ to the natural inclusion
$\iota:\mathcal M\hookrightarrow\mathcal M_\delta.$

We now apply the standard Lusternik--Schnirelmann category argument.
Assume that $\Phi_\varepsilon(\mathcal M) \subset A_1\cup\cdots\cup
A_k,$ where each $A_j$ is closed in $\Phi_\varepsilon(\mathcal M)$
and contractible in $\mathcal A_\varepsilon$. Since
\[
\beta_\varepsilon(\mathcal A_\varepsilon) \subset \mathcal M_\delta,
\]
the corresponding contractions, composed with $\beta_\varepsilon$,
show that the sets $\Phi_\varepsilon^{-1}(A_j), \qquad
j=1,\ldots,k,$ are contractible in $\mathcal M_\delta$, in the sense
relevant to the relative category of $\mathcal M$ in $\mathcal
M_\delta$. Because
\[
\mathcal M = \bigcup_{j=1}^k \Phi_\varepsilon^{-1}(A_j),
\]
we obtain $\operatorname{cat}_{\mathcal M_\delta}(\mathcal M) \leq
k.$ Taking the minimum over all such categorical coverings of
$\Phi_\varepsilon(\mathcal M)$ gives
\[
\operatorname{cat}_{\mathcal A_\varepsilon}
\bigl(\Phi_\varepsilon(\mathcal M)\bigr) \geq
\operatorname{cat}_{\mathcal M_\delta}(\mathcal M).
\]

Finally, since $\Phi_\varepsilon(\mathcal M) \subset \mathcal
A_\varepsilon,$ the monotonicity of the Lusternik--Schnirelmann
category yields
\[
\operatorname{cat}(\mathcal A_\varepsilon) \geq
\operatorname{cat}_{\mathcal A_\varepsilon}
\bigl(\Phi_\varepsilon(\mathcal M)\bigr).
\]
Consequently,
\[
\operatorname{cat}(\mathcal A_\varepsilon) \geq
\operatorname{cat}_{\mathcal M_\delta}(\mathcal M).
\]
The proof is complete.
\end{proof}

To apply the Lusternik--Schnirelmann theorem, we now verify the
compactness condition on the low-energy part of the natural Pohozaev
constraint.

\begin{proposition}\label{prop:PS_low_level}
Fix $a\in(0,a_*)$. There exists $\varepsilon_*>0$ such that, for
every $\varepsilon\in(0,\varepsilon_*)$, the functional
$J_\varepsilon|_{\mathcal P_{a,\varepsilon}}$ satisfies the
Palais-Smale condition on the sublevel
\[
\mathcal A_\varepsilon = \left\{ u\in\mathcal P_{a,\varepsilon}:
J_\varepsilon(u) \leq \varepsilon^3
\bigl(E_0(a)+h(\varepsilon)\bigr) \right\}.
\]
More precisely, if $\{u_n\}\subset\mathcal A_\varepsilon$ satisfies
$J_\varepsilon(u_n)\to c$ and $d\bigl(J_\varepsilon|_{\mathcal
P_{a,\varepsilon}}\bigr)(u_n) \to0,$ then, up to a subsequence,
\[
u_n\to u \qquad\text{strongly in }H_\varepsilon.
\]
\end{proposition}

\begin{proof}
Assume, by contradiction, that the conclusion is false. Then there
exist a sequence $\varepsilon_k\to0$ and, for every $k$, a
Palais--Smale sequence $\{u_{k,n}\}_n\subset\mathcal
A_{\varepsilon_k}$ which has no strongly convergent subsequence in
$H_{\varepsilon_k}$. By Ekeland's variational principle and a
diagonal argument, we may choose $n_k\to\infty$ and set
$v_k:=u_{k,n_k}$ so that
\[
J_{\varepsilon_k}(v_k) \leq \varepsilon_k^3 \bigl(E_0(a)+o(1)\bigr),
\]
\[
\left\| d\bigl(J_{\varepsilon_k}|_{S_{a,\varepsilon_k}}\bigr)(v_k)
\right\| \to0,
\]
and $P_{\varepsilon_k}(v_k)=o(\varepsilon_k^3).$ Here we use the
fact that $\mathcal P_{a,\varepsilon}$ is a natural constraint on
the low-energy region.

Proposition \ref{prop:compact_localization} then provides points
$y_k\in\mathbb R^3$ and a positive normalized ground state $w_a$ of
the autonomous problem such that $\operatorname{dist}(y_k,\mathcal
M)\to0$ and
\[
v_k(y_k+\varepsilon_k x) \to w_a(x) \qquad\text{strongly in
}H^1(\mathbb R^3).
\]
After returning to the original variables, this convergence implies
the strong compactness of the corresponding low-energy sequence,
contradicting its construction.

Therefore, for every sufficiently small fixed $\varepsilon>0$, every
Palais--Smale sequence contained in $\mathcal A_\varepsilon$ admits
a strongly convergent subsequence in $H_\varepsilon$. Hence
$J_\varepsilon|_{\mathcal P_{a,\varepsilon}}$ satisfies the
Palais--Smale condition on the required low-energy sublevel.
\end{proof}
We are now in a position to combine the preceding topological and
compactness results in order to obtain multiple normalized critical
points.

\begin{proposition}\label{prop:LS_multiplicity}
For every sufficiently small $\delta>0$, there exist $a_\delta>0$
and $\varepsilon_\delta>0$ such that, for every $a\in(0,a_\delta),$
 $\varepsilon\in(0,\varepsilon_\delta),$
the functional $J_\varepsilon$ possesses at least
$\operatorname{cat}_{\mathcal M_\delta}(\mathcal M)$ distinct
normalized critical points in the low-energy region $\mathcal
A_\varepsilon$. Consequently, problem \eqref{P} admits at least
$\operatorname{cat}_{\mathcal M_\delta}(\mathcal M)$ distinct
normalized solutions.
\end{proposition}

\begin{proof}
By Proposition \ref{prop:category_estimate}, for all sufficiently
small $\varepsilon>0$,
\[
\operatorname{cat}(\mathcal A_\varepsilon) \geq
\operatorname{cat}_{\mathcal M_\delta}(\mathcal M).
\]
Moreover, by Proposition \ref{prop:PS_low_level}, the functional
$J_\varepsilon|_{\mathcal P_{a,\varepsilon}}$ satisfies the
Palais--Smale condition on the low-energy sublevel $\mathcal
A_\varepsilon$.

Therefore, the Lusternik--Schnirelmann theorem applied to
$J_\varepsilon|_{\mathcal P_{a,\varepsilon}}$ yields at least
\[
\operatorname{cat}_{\mathcal M_\delta}(\mathcal M)
\]
distinct critical points
\[
u_{\varepsilon,1},\ldots,u_{\varepsilon,k} \in\mathcal
A_\varepsilon, \qquad k\geq \operatorname{cat}_{\mathcal
M_\delta}(\mathcal M).
\]

Since $\mathcal P_{a,\varepsilon}$ is a natural constraint for
$J_\varepsilon|_{S_{a,\varepsilon}}$, each critical point of
$J_\varepsilon|_{\mathcal P_{a,\varepsilon}}$ is also a constrained
critical point of $J_\varepsilon$ on $S_{a,\varepsilon}$. Hence, for
every $j=1,\ldots,k$, there exists
$\lambda_{\varepsilon,j}\in\mathbb R$ such that
\[
J_\varepsilon'(u_{\varepsilon,j}) =
\lambda_{\varepsilon,j}u_{\varepsilon,j} \qquad \text{in
}H_\varepsilon^*.
\]
Since $u_{\varepsilon,j}\in S_{a,\varepsilon},$ we also have
\[
\int_{\mathbb R^3} |u_{\varepsilon,j}|^2\,dx = a^2\varepsilon^3.
\]
Thus each $u_{\varepsilon,j}$ is a normalized solution of problem
\eqref{P}. Consequently, problem \eqref{P} admits at least
$\operatorname{cat}_{\mathcal M_\delta}(\mathcal M)$ distinct
normalized solutions. The proof is complete.
\end{proof}

We finally describe the concentration behavior of the low-energy
normalized critical points obtained above. The result follows
directly from the compactness and localization principle established
in Proposition \ref{prop:compact_localization}.

\begin{corollary}\label{cor:concentration_profiles}
Let $\varepsilon_n\to0$ and let $u_n\in S_{a,\varepsilon_n}$ be a
sequence of positive normalized critical points of
$J_{\varepsilon_n}|_{S_{a,\varepsilon_n}}$ satisfying
\[
J_{\varepsilon_n}(u_n) \leq \varepsilon_n^3\bigl(E_0(a)+o(1)\bigr).
\]
Then there exist a sequence $\{y_n\}\subset\mathbb R^3$ and a
positive normalized ground state $w_a$ of the autonomous limit
problem \eqref{P0} such that, up to a subsequence,
\[
\operatorname{dist}(y_n,\mathcal M)\to0, \qquad V(y_n)\to V_0,
\]
and
\[
u_n(y_n+\varepsilon_n x) \to w_a(x) \qquad \text{strongly in
}H^1(\mathbb R^3).
\]
\end{corollary}

\begin{proof}
Since each $u_n$ is a constrained critical point of
$J_{\varepsilon_n}$ on $S_{a,\varepsilon_n}$, there exists
$\lambda_n\in\mathbb R$ such that $J_{\varepsilon_n}'(u_n) =
\lambda_nu_n$ in $H_{\varepsilon_n}^*.$ Hence $\{u_n\}$ is a
Palais--Smale sequence for
$J_{\varepsilon_n}|_{S_{a,\varepsilon_n}}$. Moreover, Lemma
\ref{lem:scaling_pohozaev} yields $P_{\varepsilon_n}(u_n)=0$ for
every $n.$ Therefore, all the assumptions of Proposition
\ref{prop:compact_localization} are satisfied. Consequently, there
exist points $y_n\in\mathbb R^3$ and a positive normalized ground
state $w_a$ of the autonomous limit problem \eqref{P0} such that, up
to a subsequence,
\[
\operatorname{dist}(y_n,\mathcal M)\to0, \qquad V(y_n)\to V_0,
\]
and
\[
u_n(y_n+\varepsilon_n x) \to w_a(x) \qquad \text{strongly in
}H^1(\mathbb R^3).
\]
The proof is complete.
\end{proof}
We now show that their global maximum points exhibit the same
concentration behavior and, after rescaling, converge to ground
states of the autonomous limit problem.

\begin{lemma}\label{lem:concentration_maxima}
Let $\varepsilon_n\to0$ and let $\{u_n\}$ be a sequence of positive
solutions obtained in Proposition \ref{prop:LS_multiplicity}. Assume
that $J_{\varepsilon_n}(u_n) \leq
\varepsilon_n^3\bigl(E_0(a)+o(1)\bigr).$ If $x_n\in\mathbb R^3$ is a
global maximum point of $u_n$, then
$\operatorname{dist}(x_n,\mathcal M)\to0.$ Consequently, $V(x_n)\to
V_0.$ Moreover, after passing to a subsequence, there exists a
positive ground state $w$ of the autonomous problem \eqref{P0} such
that
\[
u_n(x_n+\varepsilon_n x) \to w(x) \qquad \text{strongly in
}H^1(\mathbb R^3).
\]
\end{lemma}

\begin{proof}
By Corollary \ref{cor:concentration_profiles}, there exist points
$y_n\in\mathbb R^3$ and a positive normalized ground state $w_a$ of
the autonomous problem \eqref{P0} such that
\[
\operatorname{dist}(y_n,\mathcal M)\to0, \qquad V(y_n)\to V_0,
\]
and $v_n(x):= u_n(y_n+\varepsilon_n x) \to w_a(x) \qquad
\text{strongly in }H^1(\mathbb R^3).$ By local elliptic regularity
and the equation satisfied by $v_n$, the above convergence improves,
up to a subsequence, to $v_n\to w_a \qquad \text{in
}C_{\mathrm{loc}}^{1,\alpha}(\mathbb R^3)$ for some
$\alpha\in(0,1)$. Since $w_a>0$ in $\mathbb R^3$, there exist $R>0$
and $\sigma>0$ such that $\max_{B_R(0)}w_a\geq2\sigma.$ Hence, for
all sufficiently large $n$, $\max_{B_R(0)}v_n\geq\sigma.$
Equivalently, $\max_{B_{R\varepsilon_n}(y_n)}u_n\geq\sigma.$ Since
$x_n$ is a global maximum point of $u_n$, it follows that
$u_n(x_n)\geq\sigma.$ We claim that
$\left|\frac{x_n-y_n}{\varepsilon_n}\right|$ is bounded. Suppose, by
contradiction, that $\left|\frac{x_n-y_n}{\varepsilon_n}\right|
\to+\infty.$ The uniform decay of the translated solutions gives
$v_n(x)\to0 \qquad \text{uniformly as }|x|\to+\infty,$ uniformly
with respect to $n$. Therefore,
\[
u_n(x_n) = v_n\left(\frac{x_n-y_n}{\varepsilon_n}\right) \to0,
\]
which contradicts $u_n(x_n)\geq\sigma.$ Thus, up to a subsequence,
\[
\xi_n:= \frac{x_n-y_n}{\varepsilon_n} \to\xi \qquad \text{for some
}\xi\in\mathbb R^3.
\]
Consequently, $|x_n-y_n| = \varepsilon_n|\xi_n| = O(\varepsilon_n).$
Therefore,
\[
\operatorname{dist}(x_n,\mathcal M) \leq |x_n-y_n| +
\operatorname{dist}(y_n,\mathcal M) \to0.
\]
Since $V$ is continuous and $V=V_0 \qquad \text{on }\mathcal M,$ we
obtain $V(x_n)\to V_0.$ Finally, for every $x\in\mathbb R^3$,
$u_n(x_n+\varepsilon_n x) = v_n(x+\xi_n).$ Since $v_n\to w_a \qquad
\text{strongly in }H^1(\mathbb R^3)$ and $\xi_n\to\xi,$ the
continuity of translations in $H^1(\mathbb R^3)$ yields
\[
u_n(x_n+\varepsilon_n x) \to w_a(x+\xi) \qquad \text{strongly in
}H^1(\mathbb R^3).
\]
Setting $w(x):=w_a(x+\xi),$ and using the translation invariance of
the autonomous problem \eqref{P0}, we conclude that $w$ is again a
positive ground state of \eqref{P0}. The proof is complete.
\end{proof}

\subsection{Proof of Theorem \ref{T2}}

Let $\delta>0$ be sufficiently small and choose $a_\delta>0$ and
$\varepsilon_\delta>0$ according to Proposition
\ref{prop:LS_multiplicity}. Then, for every $a\in(0,a_\delta),$
$\varepsilon\in(0,\varepsilon_\delta),$ Proposition
\ref{prop:LS_multiplicity} yields at least
$\operatorname{cat}_{\mathcal M_\delta}(\mathcal M)$ distinct
normalized solutions of problem \eqref{P}.

Let now $\varepsilon_n\to0$ and let $\{u_n\}$ be any sequence of
these solutions. By Corollary \ref{cor:concentration_profiles}, the
solutions concentrate, after rescaling, around points approaching
$\mathcal M$. Moreover, if $x_n$ is a global maximum point of $u_n$,
Lemma \ref{lem:concentration_maxima} gives
$\operatorname{dist}(x_n,\mathcal M)\to0,$ $V(x_n)\to V_0.$ This
proves the multiplicity and concentration assertions of Theorem
\ref{T2}. \qed

\vspace{0.5em}

\noindent\textbf{Data Availability Statement.} No datasets were
generated or analyzed during the current study. Therefore, data
sharing is not applicable.

\vspace{0.5em}

\noindent\textbf{Competing Interests.} The authors declare that they
have no competing interests.

\vspace{0.5em}

\noindent\textbf{Funding}

The author received no financial support for this work.

\bigskip


\noindent \textsc{$^{1}$ Khaled Khachnaoui}\\
$^{1}$ University of Kairouan, Preparatory Institute for Engineering
Studies \\Department of
 Mathematics\\Tunisia,\\
\texttt{k{\_}khachnaoui@yahoo.com} \\
\noindent \mbox{} \\


\begin{thebibliography}{99}

\bibitem{ACO}
C. O. Alves and N. V. Thin, On existence of multiple normalized
solutions to a class of elliptic problems in whole $\mathbb{R}^N$
via Lusternik--Schnirelmann category, \emph{SIAM J. Math. Anal.}
\textbf{55} (2023), 1264--1283.

\bibitem{AJC}
C. O. Alves and C. Ji, Multiple normalized solutions to a
logarithmic Schr\"{o}dinger equation via Lusternik--Schnirelmann
category, \emph{J. Geom. Anal.} \textbf{34} (2024), Paper No. 198.

\bibitem{AmbrosettiRuiz}
A. Ambrosetti and D. Ruiz, Multiple bound states for the
Schr\"{o}dinger--Poisson problem, \emph{Commun. Contemp. Math.}
\textbf{10} (2008), 391--404.

\bibitem{Azzollini}
A. Azzollini, Concentration and compactness in nonlinear
Schr\"{o}dinger--Poisson systems, \emph{J. Differential Equations}
\textbf{249} (2010), 1746--1763.

\bibitem{AzzolliniDAprize}
A. Azzollini and T. D'Aprile, On the existence of least energy
solutions for the Schr\"{o}dinger--Poisson system, \emph{Nonlinear
Anal.} \textbf{72} (2010), 2125--2132.

\bibitem{BartschValeriola}
T. Bartsch and S. de Valeriola, Normalized solutions of nonlinear
Schr\"{o}dinger equations, \emph{Arch. Math.} \textbf{100} (2013),
75--83.

\bibitem{BartschJeanjean}
T. Bartsch and L. Jeanjean, Normalized solutions for nonlinear
Schr\"{o}dinger equations, \emph{Arch. Ration. Mech. Anal.}
\textbf{230} (2018), 1099--1167.

\bibitem{BellazziniJeanjeanLuo}
J. Bellazzini, L. Jeanjean, and T. Luo, Existence and stability of
normalized solutions for the Schr\"{o}dinger--Poisson system,
\emph{J. Funct. Anal.} \textbf{281} (2021), Paper No. 109135.

\bibitem{BenciFortunato}
V. Benci and D. Fortunato, An eigenvalue problem for the
Schr\"{o}dinger--Maxwell equations, \emph{Topol. Methods Nonlinear
Anal.} \textbf{11} (1998), 283--293.

\bibitem{BenciFortunato2002}
V. Benci and D. Fortunato, Solitary waves of the nonlinear
Klein--Gordon equation coupled with the Maxwell equations,
\emph{Rev. Math. Phys.} \textbf{14} (2002), 409--420.

\bibitem{Cazenave}
T. Cazenave, \emph{Semilinear Schr\"{o}dinger equations}, Courant
Lect. Notes Math., vol. 10, Amer. Math. Soc., Providence, RI, 2003.

\bibitem{CeramiVaira}
G. Cerami and G. Vaira, Positive solutions for some non-autonomous
Schr\"{o}dinger--Poisson systems, \emph{J. Differential Equations}
\textbf{248} (2010), 521--543.

\bibitem{CeramiVaira2012}
G. Cerami and G. Vaira, Positive solutions for some nonlocal
problems, \emph{J. Differential Equations} \textbf{252} (2012),
2620--2647.

\bibitem{CST}
S. Chen and X. Tang, A comprehensive review on the existence of
normalized solutions for four classes of nonlinear elliptic
equations, \emph{Opuscula Math.} \textbf{45} (2025), 739--763.

\bibitem{CingolaniSecchi}
S. Cingolani and S. Secchi, Semiclassical states for the nonlinear
Schr\"{o}dinger equation with magnetic field, \emph{J. Math. Anal.
Appl.} \textbf{332} (2007), 1295--1313.

\bibitem{Diamagnetic}
H. L. Cycon, R. G. Froese, W. Kirsch, and B. Simon,
\emph{Schr\"{o}dinger operators with application to quantum
mechanics and global geometry}, Springer, Berlin, 1987.

\bibitem{DAprileMugnai}
T. D'Aprile and D. Mugnai, Solitary waves for nonlinear
Klein--Gordon--Maxwell and Schr\"{o}dinger--Maxwell equations,
\emph{Proc. Roy. Soc. Edinburgh Sect. A} \textbf{134} (2004),
893--906.

\bibitem{delPinoFelmer}
M. del Pino and P. L. Felmer, Local mountain passes for semilinear
elliptic problems in unbounded domains, \emph{Calc. Var. Partial
Differential Equations} \textbf{4} (1996), 121--137.

\bibitem{EstebanLions}
M. J. Esteban and P. L. Lions, Stationary solutions of nonlinear
Schr\"{o}dinger equations with an external magnetic field, in
\emph{Partial Differential Equations and the Calculus of
Variations}, Vol. I, Birkh\"{a}user, Boston, 1989, pp. 401--449.

\bibitem{FanLiTang}
S. Fan, G.-D. Li and C.-L. Tang, Normalized ground state solutions
for critical growth Schr\"{o}dinger equations with Hardy potential,
\emph{Proc. Roy. Soc. Edinburgh Sect. A}, published online, 2024.

\bibitem{FX}
X. Feng, Ground state solutions for a class of
Schr\"{o}dinger--Poisson systems with partial potential, \emph{Z.
Angew. Math. Phys.} \textbf{71} (2020), 16.

\bibitem{VR}
Q. Gao, X. He, and V. D. R\u{a}dulescu, Solutions with prescribed
mass for critical Schr\"{o}dinger--Poisson systems concentrating at
a potential well, \emph{Z. Angew. Math. Phys.} \textbf{77} (2026),
160.

\bibitem{GQH}
Q. Gao and X. He, Normalized solutions for the Choquard equations
with critical nonlinearities, \emph{Adv. Nonlinear Anal.}
\textbf{13} (2024), Paper No. 20240030.

\bibitem{GZZ}
Z. Guo and T. Zhang, Normalized solutions for Choquard equations
involving Hardy terms and local perturbations, \emph{Period. Math.
Hungar.} \textbf{91} (2025), 169--198.

\bibitem{HeMengSquassina2024}
X. He, Y. Meng and M. Squassina, Normalized solutions for a
fractional Schr\"{o}dinger--Poisson system with critical growth,
\emph{Calc. Var. Partial Differential Equations} \textbf{63} (2024),
Article No. 142.

\bibitem{HXL}
X. He, W. Liu, and Y. Meng, Normalized solutions for the critical
Schr\"{o}dinger--Poisson systems with potentials, \emph{Adv. Differ.
Equ.} \textbf{31} (2026), 459--512.

\bibitem{HXX}
X. He, Y. Meng, and V. D. R\u{a}dulescu, Prescribed mass solutions
for Schr\"{o}dinger--Poisson systems with combined critical
nonlinearities, preprint (2024).

\bibitem{JeanjeanLu}
L. Jeanjean and S. S. Lu, On the existence of multiple normalized
solutions for a Schr\"{o}dinger--Poisson system, \emph{Calc. Var.
Partial Differential Equations} \textbf{59} (2020), Paper No. 25.

\bibitem{JPY}
P. Jin, H. Yang, and X. Zhou, Normalized solutions for
Schr\"{o}dinger equations with critical Sobolev exponent and
perturbations of Choquard terms, \emph{Bull. Math. Sci.} \textbf{15}
(2025), 36.

\bibitem{JZZ}
Z. Jin and W. Zhang, Normalized solutions for nonlinear
Schr\"{o}dinger equation involving potential and Sobolev critical
exponent, \emph{J. Math. Anal. Appl.} \textbf{535} (2024), 128161.

\bibitem{KangTang}
J.-C. Kang and C.-L. Tang, Normalized solutions for the nonlinear
Schr\"{o}dinger equation with potential and combined nonlinearities,
\emph{Results Math.} \textbf{80} (2025), Article No. 74.

\bibitem{KK}
K. Khachnaoui, On the fractional Schr\"{o}dinger equations with
critical nonlinearity, \emph{Results Math.} \textbf{78} (2023), 68.

\bibitem{KKa}
K. Khachnaoui, Multiple equilibrium states in a nonlocal medium with
Kirchhoff-type tension law, \emph{Expositiones Mathematicae} (2026),
125793.

\bibitem{KKK} K. Khachnaoui, Normalized Semiclassical Solutions to Magnetic Schr\"odinger-Poisson Systems with Critical Local and Nonlocal Interactions, \emph{arXiv:2607.08381},
2026.

\bibitem{LSM}
S. Lancelotti and R. Molle, Normalized positive solutions for
Schr\"{o}dinger equations with potentials in unbounded domains,
\emph{Proc. Roy. Soc. Edinb. Sect. A} \textbf{153} (2023),
1023--1045.

\bibitem{LiLiTang}
G.-D. Li, Y.-Y. Li and C.-L. Tang, Ground state solutions for
critical Schr\"{o}dinger equations with Hardy potential,
\emph{Nonlinearity} \textbf{35} (2022), 5076--5108.

\bibitem{LiZou}
H. Li and W. Zou, Normalized ground state for the Sobolev critical
Schr\"{o}dinger equation involving Hardy term with combined
nonlinearities, \emph{Math. Nachr.} \textbf{296} (2023), 2440--2466.

\bibitem{LEH}
E. H. Lieb, Thomas-Fermi and related theories and molecules,
\emph{Rev. Mod. Phys.} \textbf{53} (1981), 603--641.

\bibitem{LEL}
E. Lieb and M. Loss, \emph{Analysis}, Graduate Studies in
Mathematics, vol. 14, Amer. Math. Soc., Providence, RI, 2001.

\bibitem{Lions1}
P. L. Lions, The concentration-compactness principle in the calculus
of variations. The locally compact case, Part I, \emph{Ann. Inst. H.
Poincar\'{e} Anal. Non Lin\'{e}aire} \textbf{1} (1984), 109--145.

\bibitem{Lions2}
P. L. Lions, The concentration-compactness principle in the calculus
of variations. The locally compact case, Part II, \emph{Ann. Inst.
H. Poincar\'{e} Anal. Non Lin\'{e}aire} \textbf{1} (1984), 223--283.

\bibitem{Lions3}
P. L. Lions, The concentration-compactness principle in the calculus
of variations. The limit case, Part I, \emph{Rev. Mat. Iberoam.}
\textbf{1} (1985), 145--201.

\bibitem{LPL}
P. L. Lions, Solutions of Hartree-Fock equations for Coulomb
systems, \emph{Commun. Math. Phys.} \textbf{109} (1987), 33--97.

\bibitem{LLF}
L. Long and X. Feng, Normalized solutions to a class of
Choquard-type equations with potential, \emph{Topol. Methods
Nonlinear Anal.} \textbf{63} (2024), 515--536.

\bibitem{MengHe2023}
Y. Meng and X. He, Normalized solutions for the
Schr\"{o}dinger--Poisson system with doubly critical growth,
\emph{Topol. Methods Nonlinear Anal.} \textbf{62} (2023), 581--615.

\bibitem{MY}
Y. Meng and X. He, Normalized solutions for the
Schr\"{o}dinger--Poisson system with doubly critical growth,
\emph{Topol. Methods Nonlinear Anal.} \textbf{62} (2023), 509--534.

\bibitem{MengHe2024}
Y. Meng and X. He, Multiplicity of normalized solutions for the
fractional Schr\"{o}dinger--Poisson system with doubly critical
growth, \emph{Acta Math. Sci. Ser. B (Engl. Ed.)} \textbf{44}
(2024), 997--1019.

\bibitem{MYX}
Y. Meng, X. He, and P. Winkert, Multiple normalized solutions for
$L^2$-mass supercritical Choquard equations concentrating at a
potential well, \emph{Milan J. Math.} \textbf{94} (2026), 319--367.

\bibitem{RuizSemiclassical}
D. Ruiz, Semiclassical states for coupled Schr\"{o}dinger--Maxwell
equations: concentration around a sphere, \emph{Math. Models Methods
Appl. Sci.} \textbf{15} (2005), 141--164.

\bibitem{Ruiz}
D. Ruiz, The Schr\"{o}dinger--Poisson equation under the effect of a
nonlinear local term, \emph{J. Funct. Anal.} \textbf{237} (2006),
655--674.

\bibitem{SMW}
M. Shu, L. Wen, and H. Yang, Normalized solutions for planar
Schr\"{o}dinger--Poisson equation with mixed nonlinearities,
\emph{Bull. Math. Sci.} \textbf{15} (2025), 24.

\bibitem{Szulkin}
A. Szulkin, Ljusternik-Schnirelmann theory on $C^1$-manifolds,
\emph{Ann. Inst. H. Poincar\'{e} Anal. Non Lin\'{e}aire} \textbf{5}
(1988), 119--139.

\bibitem{Terracini}
S. Terracini, On positive entire solutions to a class of equations
with a singular coefficient and critical exponent, \emph{Adv.
Differential Equations} \textbf{1} (1996), 241--264.

\bibitem{WS}
L. Wei and Y. Song, Normalized solutions for critical
Schr\"{o}dinger equations involving $(2,q)$-Laplacian, \emph{Opusc.
Math.} \textbf{45} (2025), 685--716.

\bibitem{WMM}
M. Willem, \emph{Minimax Theorems}, Birkh\"{a}user, Boston, 1996.

\bibitem{XCJ}
Z. Xie, J. Chen, and Y. Tan, Multiple normalized solutions to
Schr\"{o}dinger equations in $\mathbb{R}^N$ with critical growth and
potential, \emph{J. Fixed Point Theory Appl.} \textbf{26} (2024),
41.

\bibitem{ZhaoZhao}
L. Zhao and F. Zhao, On the existence of solutions for the
Schr\"{o}dinger--Poisson equations, \emph{J. Math. Anal. Appl.}
\textbf{346} (2008), 155--169.

\end{thebibliography}
\end{document}